\documentclass[11pt]{amsart}

\usepackage{amsfonts, bbm}
\usepackage{amssymb}
\usepackage{amsmath, amsthm, amsrefs, microtype}

\usepackage{graphicx}
\usepackage[all]{xy}
\usepackage[margin=1.5in]{geometry}

\usepackage[shortlabels]{enumitem} \setlist{wide}

\usepackage{hyperref}

\usepackage{lmodern}

\usepackage{libertine}
\usepackage{libertinust1math}

\usepackage[T1]{fontenc}

\newtheorem{theorem}{Theorem}[section]

\newtheorem{theoremA}{Theorem}

\newtheorem{lemma}[theorem]{Lemma}
\newtheorem{corollary}[theorem]{Corollary}

\theoremstyle{remark}
\newtheorem{remark}[theorem]{Remark}

\numberwithin{equation}{section}

\newcommand{\T}{\mathbb{T}}
\newcommand{\C}{\mathbb{C}}
\newcommand{\R}{\mathbb{R}}
\newcommand{\Q}{\mathbb{Q}}
\newcommand{\Z}{\mathbb{Z}}
\newcommand{\N}{\mathbb{N}}
\newcommand{\Ntwo}{\mathcal{N}_2}
\newcommand{\her}{\mathrm{her}}
\newcommand{\sa}{\mathrm{sa}}
\newcommand{\el}{\mathrm{el}}

\newcommand{\ZZ}{\mathcal{D}}

\newcommand{\Cu}{\mathrm{Cu}}

\newcommand{\rank}{\mathrm{rank}}

\title[Simplicity, bounded normal generation, and automatic continuity]{Simplicity, bounded normal generation, and automatic continuity of groups of unitaries}
\author{Abhinav Chand}
\author{Leonel Robert}

\begin{document}
\begin{abstract}
We show that the commutator subgroup of the group of unitaries connected to the identity in a simple unital C*-algebra is simple modulo its center. We then go on to investigate the role of ``regularity properties'' in the structure of the special unitary group of a C*-algebra. Under mild assumptions, we show that this group has the invariant automatic continuity property and a unique polish group topology. Strengthening our assumptions in the case of simple C*-algebras, we show that the special unitary group modulo its center has bounded normal generation. These results apply to all simple purely infinite C*-algebras and too all simple nuclear C*-algebras in the ``classifiable class''. We show with counterexamples how our conclusions may in general fail if no regularity conditions are imposed
on the C*-algebra.
\end{abstract}

\maketitle

\section{Introduction}

Let $A$ be a unital C*-algebra. Let $U_0(A)$ denote the connected component of the identity of the unitary group of $A$. Let $DU_0(A)$ denote the commutator (or derived) subgroup
of $U_0(A)$. If $A$ is the C*-algebra of $n\times n$ matrices $M_n(\C)$, then $DU_0(A)$---i.e., the special unitary group $SU(n)$---is perfect and simple modulo its center. We are thus led to ask whether these properties hold in more generality. If $A$ has no 1-dimensional representations, then indeed $DU_0(A)$ is perfect (\cite[Theorem 6.2]{RobertNormal}). On the other hand, since non-commutative closed two-sided ideals of $A$ give rise to non-central normal subgroups of $DU_0(A)$, for the simplicity question it is natural to assume that $A$ itself be a simple C*-algebra.  This question, and its counterpart with invertible elements instead of unitaries,  has been investigated by several authors, going back to Kadison in the 50s, working in the setting of von Neumann algebra factors, followed by de la Harpe, Skandalis, Thomsen, Ng, Ruiz, among others (see \cite{kad1, delaHarpe, delaHarpe-Skandalis3, thomsen, ng-ruiz}).  In\cite{RobertNormal} the second author settled the question of simplicity modulo the center in the setting  of invertible elements: the commutator subgroup of the connected component of the identity of the group of invertible elements of a simple unital C*-algebra is simple modulo its center. In the case of unitaries, the methods used in \cite{RobertNormal} also indicate a path to a proof, but key technical questions on Lie ideals were left unanswered in \cite{RobertNormal}. 

The starting point of this paper is to address  the  technical points left unsolved in \cite{RobertNormal}. First, working in the additive setting, we prove that for a simple unital C*-algebra $A$,  a 
noncentral  abelian subgroup $V$ of $iA_{\sa}$ (the skewadjoint elements of $A$) that is invariant under conjugation by $DU_0(A)$ must contain $[A,A]\cap iA_{\sa}$ (the skewadjoint elements in the linear span of the commutators). This partially answers a question raised by Marcoux in \cite[Question 4]{marcoux09}. Then, using techniques from \cite{RobertNormal}, we pass from the  ``additive''   to the ``multiplicative'' setting to prove the following theorem.

\begin{theoremA}\label{DUsimple}
If $A$ is simple and unital then $DU_0(A)$ is simple modulo its center.
\end{theoremA}

While proving these results we take care to establish much finer versions that apply to non-simple C*-algebras, and that have a  quantitative nature (phrased in terms of generating sets of normal subgroups). Our aim is to apply these finer results to questions on automatic continuity and bounded normal generation. In doing this, we follow in the footsteps of parallel results obtained by Dowerk and Thom for von Neumann algebra factors in \cite{DowerkThom, DowerkThom2}. These  applications still revolve around $DU_0(A)$, but in situations where it naturally carries a Banach-Lie group structure, as it agrees  with the group  $SU_0(A)$. 

We denote by  $SU_0(A)$ the subgroup of $U_0(A)$ generated by  $\{e^{ih}: h\in \overline {[A,A]}\cap A_{\sa}\}$. 
It is not difficult to show that $SU_0(A)$ is a path connected normal subgroup of $U_0(A)$ containing $DU_0(A)$ and contained in the kernel of the de la Harpe-Skandalis determinant. We regard $SU_0(A)$ as the Banach-Lie subgroup of $U_0(A)$ associated to the Lie algebra of skewadjoint commutators 
$\overline{[A,A]}\cap iA_{\sa}$.  As such,  $SU_0(A)$ comes endowed with the topology having for  basis of neighborhoods at the identity the sets
\[
\{e^{ih}:h\in \overline{[A,A]}\cap A_{\sa}, \, \|h\|<\epsilon\} \quad  (\epsilon>0).
\] 
The exponential length on $SU_0(A)$ is defined as
\[
\el_{SU_0(A)}(u)=\Big\{\inf \sum_{i=1}^n \|h_i\|:u=\prod_{j=1}^n e^{h_j}\hbox{ and }h_j\in \overline{[A,A]}\cap A_{\sa}\hbox{ for all }j\Big\}.
\]		
This gives rise to a metric $(u,v)\mapsto \el(u^*v)$ that is left and right invariant and induces the topology on $SU_0(A)$ (\cite[Theorem A]{ando}). Note: the topology
on $SU_0(A)$ is in general not the same as the topology induced by its inclusion in $U_0(A)$ (where the topology on the latter is that induced by the norm). In fact, $SU_0(A)$ can be dense in $U_0(A)$ in the norm topology, e.g., if $A$ is  simple, infinite dimensional, and of real rank zero.

Under the assumption that $[A,A]$ is a closed set, which  is present for many C*-algebras, we have  that $DU_0(A)=SU_0(A)$, and so $DU_0(A)$
is a complete metric space. 

Let us say that a C*-algebra $A$ has \emph{bounded commutators generation} if $[A,A]$ is closed, and furthermore there exist $m\in \N$ and $C>0$ such that 
if $h\in \overline{[A,A]}$ then 
\[
h=\sum_{i=1}^m [x_i,y_i]
\]
for some $x_j,y_j\in A$ such that $\|x_j\|\cdot \|y_j\|\leq C\|h\|$ for all $j$.

Bounded commutators generation is not an uncommon property of C*-algebras.
We shall return to this point later.
But first, let us state our automatic continuity result. 

Following Dowerk and Thom (\cite[Definition 8.8]{DowerkThom}), let us say that a topological group $G$ has the invariant automatic continuity property if any group homomorphism
$\phi\colon G\to H$, where $H$ is a topological separable SIN group, is continuous.  (A SIN group is one with a basis of neighborhoods  of the identity invariant under conjugation.)

\begin{theoremA}\label{mainautomatic}
Let $A$ be a unital C*-algebra with bounded  commutators  generation and a full square zero element (i.e., $x\in A$ such that $x^2=0$ and the ideal generated by $x$ is $A$). The following are true:
\begin{enumerate}[(i)]
\item
$SU_0(A)$ has the invariant automatic continuity property.
\item
If $A$ is separable, then $SU_0(A)$ admits a unique polish group topology.
\end{enumerate} 
\end{theoremA}

Following Dowerk and Thom (\cite[Definition 2.1]{DowerkThom}), let us say that a simple group $G$ has bounded normal generation (BNG) if for any $g\in G\backslash \{1\}$ there exists $n\in \N$ such that  $G=(\{hg^{\pm 1}h^{-1}:h\in G\})^n$. 

We introduce here a weaker version of the BNG property for topological groups. Let us say that a topological group  $G$ has local bounded normal generation if for any $g\in G\backslash \{1\}$ there exists $n\in \N$ such that  $(\{hg^{\pm 1}h^{-1}:h\in G\})^n$
is a neighborhood of the identity.

\begin{theoremA}\label{mainBNG}
Let $A$ be a simple unital  C*-algebra with bounded commutators  generation. The following are true:
\begin{enumerate}[(i)]
\item
$SU_0(A)/Z(SU_0(A))$ is a simple group  with local BNG.
\item
If $SU_0(A)$ is bounded as a metric space, then $SU_0(A)/Z(SU_0(A))$ has BNG.
\end{enumerate} 
\end{theoremA}

If $A$ is a traceless unital C*-algebra, then by a theorem of Pop $A=[A,A]$ and $A$ has bounded commutators generation. It follows that in this case 
$DU_0(A)=SU_0(A)=U_0(A)$,
and Theorems \ref{mainautomatic} and \ref{mainBNG} apply directly to $U_0(A)$ (endowed with the norm topology). Let us single out the well studied case of  purely infinite simple C*-algebras:   If $A$ is a unital, separable, simple, purely infinite C*-algebra, then $U_0(A)$ has the automatic invariant continuity property and a unique polish group topology. Moreover, since the exponential length of  simple purely infinite C*-algebras is finite (\cite{phillips}), $U_0(A)/\T$ is a simple group with bounded normal generation.

Combining the invariant automatic continuity property  with results on the structure of Lie algebra isomorphisms by Ara-Mathieu and Bre\v{s}ar, we show that the group structure of $U_0(A)$ suffices as a classification invariant for prime traceless C*-algebras. More specifically, if $A$ and $B$ are separable prime traceless unital C*-algebras containing full square zero elements, and $U_0(A)\cong U_0(B)$ as groups, then $A$ is either isomorphic or anti-isomorphic to $B$ (Corollary \ref{primeclassification}).  The question whether C*-algebras can be classified up to (anti-)isomorphism by their unitary groups has been studied in \cite{giordano}, but it remains in general not well understood.

In contrast to the traceless case, if $A$ has at least one tracial state, then $U_0(A)$ has discontinuous automorphisms. This is well-known in the case of the classical unitary group (i.e. $A=M_n(\C)$), and the argument can be adapted to an arbitrary C*-algebra with a tracial state (Theorem \ref{discontinuousauto}). So we must turn to $SU_0(A)$ in this case. In order to apply Theorems \ref{mainautomatic} and \ref{mainBNG}, we must 
ensure that the C*-algebra has the property of bounded commutators generation. 
For unital C*-algebras with tracial states, this property is known to hold assuming some form of regularity on the C*-algebra. Here are some classes of C*-algebras with bounded commutators generation (see Theorem \ref{BCGclasses}):
\begin{itemize}
\item
C*-algebras of  finite decomposition rank (in the sense of Kirchberg and Winter) and no finite dimensional representations.
\item
C*-algebras with strict comparison of positive elements by traces and almost divisible Cuntz semigroup.
\item
Exact C*-algebras tensorially absorbing the Jiang-Su algebra.
\end{itemize}
Theorems \ref{mainautomatic} and \ref{mainBNG} (i) can be readily applied to the C*-algebras in these classes (the existence of full square zero elements holds automatically in each case).

By Theorem \ref{mainBNG} (ii), in order to guarantee BNG we must verify the boundedness of  $SU_0(A)$ (which is also necessary to have BNG). 
We prove below that if $A$ is a simple C*-algebra with strict comparison of positive elements by traces, stable rank one, and finite exponential rank, then $SU_0(A)$ is bounded. 
All  necessary hypotheses are  present for the simple nuclear C*-algebras  recently classified in the Elliott program: If $A$ is a separable simple unital C*-algebra with finite nuclear dimension and in the UCT class, then $SU_0(A)$ has the automatic invariant continuity property and a unique polish group topology, and $SU_0(A)$ modulo its center  is a simple group with BNG (Corollary \ref{classifiable}).

In the final section of the paper we turn to the construction of counterexamples. These show that bounded commutators generation cannot be altogether dropped  in Theorem \ref{mainBNG}.  In \cite{RobertCommutators} the second author constructed examples of simple unital C*-algebras without bounded commutators generation. Essentially the same construction allows us to prove the following theorem.

\begin{theoremA}\label{counterexamples}
There exists a simple unital AH C*-algebra $A$ with a unique tracial state and the following properties:
\begin{enumerate}[(i)]
\item
$DU_0(A)/Z(DU_0(A))$ is a simple group without BNG. 
\item
The inclusion of $DU_0(A)$ in $SU_0(A)$ is proper.
\item
There exist exponentials $e^{ih_1},e^{ih_2},\ldots\in SU_0(A)$ such that $\el_{U_0(A)}(e^{ih_n})\to \infty$,
where $\el_{U_0(A)}(\cdot)$ denotes the exponential length in $U_0(A)$. In particular,
 $SU_0(A)$ is unbounded as a metric space.
\end{enumerate}
\end{theoremA}

To prove this theorem we first find topological obstructions to ``short exponential length on commutators'' in the unitary groups of homogeneous C*-algebras. A simple C*-algebra is then obtained putting the result for homogeneous C*-algebras through the Villadsen machine. We note  that, although a general recipe for obtaining  C*-algebras of infinite exponential length and having more than one tracial state was provided in \cite{phillips}, examples with a unique tracial state are harder to come by and did not previously exist in the literature. Theorem \ref{counterexamples}  shows  that \cite[Question 4.4]{phillips} has a negative answer.

\section{Preliminaries}
Let us introduce some of the notation that will be used throughout the paper. 

 Let $A$ be a unital C*-algebra. Let $A_{\sa}$ and $A_+$ denote the sets of selfadjoint and positive elements of $A$, respectively.
Let $GL(A)$ and $U(A)$ denote the groups of invertible and unitary elements of $A$. Let  $GL_0(A)$ and $U_0(A)$ denote
the connected component of the identity in $GL(A)$ and $U(A)$. The commutator, or derived, subgroups of $GL_0(A)$ and $U_0(A)$ are denoted by $DGL_0(A)$
and $DU_0(A)$. 

Given $x,y\in A$, let  us denote their commutator $xy-yx$ by $[x,y]$. The notation $[X,Y]$, when applied to subsets $X,Y$ of $A$, denotes the linear span of all $[x,y]$
with $x\in X$ and $y\in Y$. Thus, $[A,A]$ denotes the linear span of all the commutator elements of $A$. 

In the remainder of this section we prove a number of technical results that will be needed in the 
next section.

Let us recursively define polynomials in noncommuting variables (nc-polynomials) $\pi_n$, for $n\geq 0$, by $\pi_0(x)=x$ and 
\begin{equation}\label{ncpolys}
\pi_n(x_1,\ldots,x_{2^n})=[\pi_{n-1}(x_1,\ldots,x_{2^{n-1}}),\pi_{n-1}(x_{2^{n-1}+1},\ldots,x_{2^n})],
\end{equation}
for $n\geq 1$.

\begin{lemma}\label{robnormal}
In an associative algebra, the nc-polynomial $[a\pi_3(x_1,\ldots,x_8)b,c]$ (in the variables $a,b,c,x_1,\ldots, x_8$) is a sum of terms of the forms
\begin{enumerate}
\item[(a)]
$[x_i,[p_1,p_2]]$, where $p_1,p_2$ are nc-polynomials on $a,b,c,x_1,\ldots,x_8$,
\item[(c)]
$[[x_i,[q_1,q_2]],[r_1,r_2]]$, where $q_1,q_2,r_1,r_2$ are nc-polynomials on $a,b,c,x_1,\ldots,x_8$.
\end{enumerate}
\end{lemma}

\begin{proof}
This is \cite[Lemma 4.1]{RobertNormal}. 
\end{proof}

We call a set $X\subseteq A$ full if the closed two-sided ideal that it generates is $A$. We call $X$ fully noncentral if it is not contained in the center of $A$, nor it is mapped onto the center of any non-trivial quotient of $A$. Put differently, $X$ is fully noncentral if $[X,A]$ is a full set.

\begin{lemma}\label{fullpi3}
Let $X$ be a fully noncentral subset of $A$ invariant under conjugation by $DU_0(A)$. Then $\pi_3(X^8)$ is a full subset of $A$.
\end{lemma}

\begin{proof} Let $W$ denote the closed linear span of $X$. Let us show that $W$ is a Lie ideal of $A$. Observe that $W$ is a fully noncentral closed subspace invariant under conjugation by $DU_0(A)$. Since
\[
\lim_{t\to 0}\frac 1{it}(e^{it[x^*,x]}he^{-it[x^*,x]}-h) = [h,[x^*,x]],
\]
and $e^{it[x^*,x]}\in DU_0(A)$ (\cite[Theorem 6.2]{RobertNormal}), we have 
that $[h,[x^*,x]]\subseteq W$ for all $h\in W$ and $x\in A$. Further, since the selfcommutators $[x^*,x]$ span $[A,A]$,  $[W,[A,A]]\subseteq W$, i.e., $W$ is a Lie ideal of $[A,A]$. By  \cite[Theorem 1.15]{RobertLie}, this implies that $W$ is a Lie ideal  of $A$. 

Let $I$ be the closed two-sided ideal generated by $\pi_3(W^8)$. Let $\widetilde{W}\subseteq A/I$ be the image of $W$ under the quotient map, so that
 $\pi_3(\widetilde{W}^8)=0$. This implies that $[A/I, \widetilde W]=0$, by (repeated applications of)  \cite[Lemma 2]{herstein}. Since $W$ is fully noncentral, we must have that $I=A$. That is,   $\pi_3(W^8)$ is a full set. Since $\pi_3$ is multilinear, the closed linear spans of $\pi_3(W^8)$ and $\pi_3(X^8)$ agree. Hence,  $\pi_3(X^8)$ is a full set.
\end{proof}

The following theorem may be extracted from the results and methods in \cite{RobertLie}, although it is not stated there explicitly.

\begin{theorem}\label{additivesim}
Let $A$ be a unital C*-algebra without 1-dimensional representations. Let $V\subseteq A$ be a fully noncentral subset invariant under conjugation by elements in $DGL_0(A)$ and under multiplication by $-1$. Then there exist $n\in \N$ and $C>0$ such that for all  $c,d\in A$, with $\|c\|\cdot \|d\|\leq 1$, 
$[c,d]\in \sum_{j=1}^n v_j$, for some $v_j\in V$ such that $\|v_j\|\leq C$ for all $j$. 
\end{theorem}

\begin{proof}
Since $V$ is fully noncentral,   $\pi_3(V^8)$ is a  full set by Lemma \ref{fullpi3}. Hence, we can write 
\[
1=\sum_{j=1}^m x_j\pi_3(v_{j,1},\ldots,v_{j,8})y_j,
\]
where $v_{j,k}\in V$ and $x_j,y_j\in A$ for all $j,k$. Enlarging the number of terms in the sum if necessary, we may assume that  $\|x_j\|,\|y_j\|\leq 1$ for all $j$. Let $L=\max_{j,k} \|v_{j,k}\|$. Let us multiply by a contraction $c\in A$ on the left and take commutator with a contraction $d\in A$, to obtain
\[
[c,d]=\sum_{j=1}^m [cx_j\pi_3(v_{j,1},\ldots,v_{j,8})y_j,d].
\] 
Applying Lemma \ref{robnormal} on the right-hand side, we expand each term $[cx_j\pi_3(v_{j,1},\ldots,v_{j,8})y_j,d]$ into a sum of terms of the forms 
$[v_{j,k},[p_1,p_2]]$ or $[[v_{j,k},[q_1,q_2],[r_1,r_2]]$. 
Here the norms of the elements $p_i,q_i,r_i$ are bounded by a universal polynomial on the constant $L$ (since they are nc-polynomials on $cx_j$, $v_{j,k}$, $y_j$, and $d$). 
By \cite[Theorem 4.3]{RobertNormal}, there exist $m_1\in \N$ and $C_1>0$ such that  each commutator $[s_1,s_2]$ is a sum of a $m_1$ square zero elements whose norms  are bounded by  $C_1\|s_1\|\cdot \|s_2\|$. Applying this to the commutators in the terms $[v_{j,k},[p_1,p_2]]$ and 
$[[v_{j,k},[q_1,q_2],[r_1,r_2]]$ we obtain
\begin{equation}\label{cdsum}
[c,d]=\sum_{j=1}^{m_2} [v_{j},z_j]+\sum_{j=1}^{m_3}[[v_j',z_j'],z_j''], 
\end{equation}
where $v_j,v_j'\in V$ for all $j$ satisfy that $\|v_j\|,\|v_j'\|\leq L$ for all $j$, where  $z_j,z_j',z_j''$ are square zero elements whose norms are bounded by a function of $L$ for all $j$, and where $m_2,m_3$ are both bounded by a universal constant  times $mm_1$.
Observe that  
\[
[v,z] = (1+z)v(1-z)+(1-z)v(1+z) - 2v.
\]
If $z$ is a square zero element then the first two terms on right-hand side are similarity conjugates of $v$, since $(1+z)^{-1}=1-z$. Further, $1+z\in DGL_0(A)$ (\cite[Lemma 2.5]{RobertNormal}), 
so the conjugation is by an element in $DGL_0(A)$. Using this on the right-hand side of \eqref{cdsum}, the terms $[v_{j},z_j]$ and $[[v_j',z_j'],z_j'']$
can be expressed as sums of similarity conjugates of  $\pm v_{j}, \pm v_j'$  by elements in $DGL_0(A)$. This is a sum of elements in $V$ whose number of terms is independent of $c$ and $d$. Further, since the norms of the square zero elements $z_j,z_j',z_j''$ are bounded by a constant (a function of $L$), the norms of the resulting elements in $V$ (conjugates of the  $\pm v_j$s and $\pm v_j'$s) are bounded by a constant independent of $c$ and $d$. This proves the theorem. 
\end{proof}

Let
\[
\Ntwo=\{x\in A:x^2=0\}.
\]
 That is, $\Ntwo$ denotes the set of square zero elements of $A$.

\begin{lemma}\label{lemmaxhx}
Let $h\in A$ and  $x\in \Ntwo$ be such that $\|x\|\leq 2$. Then $xhx$ is a sum of 8 terms in the set
\begin{equation}\label{hconjugates}
H:=\{uhu^*,\, \pm iuhu^*:u\in DU_0(A)\}.
\end{equation}

\end{lemma}	
\begin{proof}
Assume that $x\in \Ntwo$ and $\|x\|\leq 1$. We will show that $4xhx$ is a sum of 8 terms in the set \eqref{hconjugates}, 
which is essentially what is claimed in the lemma.
 
For each $\alpha\in \C$ of absolute value $1$, define
\[
u(\alpha) = \sqrt{1-(x^*x+xx^*)} + i(\alpha x+\overline{\alpha} x^*)=a+ib(\alpha).
\]
Let us show that these are unitaries in $DU_0(A)$. It suffices to work in the universal C*-algebra generated by a square zero contraction.
This is the C*-algebra $M_2(C_0(0,1])$, with square zero contraction $x=\begin{pmatrix} 0 & t\\0&0\end{pmatrix}$. In this case
\[
u(\alpha)=\begin{pmatrix} \sqrt{1-t^2} & i\alpha t\\
i\overline{\alpha}t & \sqrt{1-t^2}\end{pmatrix}=
\begin{pmatrix} \alpha & 0\\
0 & \overline{\alpha}\end{pmatrix}
\begin{pmatrix} \sqrt{1-t^2} & t\\
-t & \sqrt{1-t^2}\end{pmatrix}
\begin{pmatrix} \overline{\alpha} & 0\\
0 & \alpha\end{pmatrix}.
\] 
It suffices to write as a commutator the middle matrix on the right-hand side. The change of variables $t=\sin(\theta)$ gives rise to an isomorphism from $M_2(C_0(0,1])$
to $M_2(C_0(0,\pi/2])$. In the latter algebra, the middle matrix on the right-hand side becomes the rotation matrix
\[
\begin{pmatrix}
\cos \theta & \sin \theta\\
-\sin \theta & \cos \theta
\end{pmatrix},
\]
which is easily expressed as a commutator:
 \[\begin{pmatrix}
\cos \theta & \sin \theta\\
-\sin \theta & \cos \theta
\end{pmatrix}=
\left(
\begin{pmatrix}
\cos \theta/2 & \sin \theta/2\\
-\sin \theta/2 & \cos \theta/2
\end{pmatrix},
\begin{pmatrix}
\cos(\theta) & \sin(\theta)\\
\sin(\theta) &  -\cos(\theta)
\end{pmatrix}
\right),
\]
for all $0\leq \theta\leq \pi/2$


Back in the C*-algebra $A$, we have that 
\[
u(\alpha)^*hu(\alpha)+u(\alpha)hu(\alpha)^*=aha+b(\alpha)hb(\alpha).
\]
The left-hand side is a sum of two element in the set  $H$ (defined in \eqref{hconjugates}), and $a$ and $b$ are as in the definition of $u(\alpha)$ above. 
Subtracting two expressions of the form $aha+b(\alpha)hb(\alpha)$  for  two different values of $\alpha$ we get that
\[
(\alpha_1 x+\overline{\alpha_1} x^*)h(\alpha_1 x+\overline{\alpha_1} x^*) - 
(\alpha_2 x+\overline{\alpha_2} x^*)h(\alpha_2 x+\overline{\alpha_2} x^*)
\]
is a sum of 4 elements of the set $H$.
Using that $\alpha_1$ and $\alpha_2$ are in the unit circle, this simplifies to
\[
(\alpha_1^2-\alpha_2^2)xhx + (\overline{\alpha_1}^2-\overline{\alpha_2}^2)x^*hx^*.
\]
Hence,  $\alpha xhx + \overline{\alpha} x^*hx^*$ is a sum of 4 elements in $H$ for any $|\alpha|\leq 2$. With $\alpha=2$
and $\alpha=2i$ we get
$
2(xhx+x^*hx^*)$  and $2i(xhx-x^*hx^*)$, respectively.
Hence, $4xhx$ as a sum of 8  elements in $H$.
\end{proof}


\begin{lemma}\label{tricky}
Let $B$ be a C*-algebra (possibly nonunital). Let $b\in B$ be of norm at most 1 such that $eb=be=b$ for some $e\in B_+$ of norm 1. Let
\[
h =\begin{pmatrix}
0 & b\\
0 & 0 
\end{pmatrix}\in M_2(B).
\]	
Let $x,y\in B$ be of norm at most 1. Then
\[
\begin{pmatrix}
0 & xb^2y\\
0 & 0 
\end{pmatrix}
\]	
is a sum of $33540$ terms in the set 
\begin{equation}\label{oftheform}
H=\{\pm uhu^*,\, \pm i uhu^*:u\in DU_0(M_2(B))\}.
\end{equation}
\end{lemma}

\begin{proof}
Regard $B$ embedded in $M_2(B)$ in the top left corner. Set $W = \begin{pmatrix} 0 & 1\\ 0 & 0\end{pmatrix}$, which we regard as an element of the multiplier algebra of $M_2(B)$. Observe that  for $x\in B$ we have that 
\[
xW=\begin{pmatrix} 0 & x\\ 0 & 0\end{pmatrix}.
\]

Let  $x\in B$ be such that   $\|x\|\leq 4$. We have that
\[
\begin{pmatrix} 0 & xbx\\ 0 & 0\end{pmatrix}=
\begin{pmatrix} 0 & x/2  \\ 0 & 0\end{pmatrix}
\begin{pmatrix} 0 & 0  \\ 2e & 0\end{pmatrix}
\begin{pmatrix} 0 & b  \\ 0 & 0\end{pmatrix}
\begin{pmatrix} 0 & 0  \\ 2e & 0\end{pmatrix}
\begin{pmatrix} 0 & x/2  \\ 0 & 0\end{pmatrix}.
\]
Applying Lemma \ref{lemmaxhx} twice we conclude that  $xbxW$ (element on the left-hand side) is a sum of $64$ terms in the set $H$ (defined in \eqref{oftheform}). Observe that this applies in particular to $b^2W=b^{\frac12}(b)b^{\frac12}W$. 

Applying the conclusion from the previous paragraph with  $x+e$ in place of $x$, where $\|x\|\leq 3$, 
we get that 
\[
(x+e)b(x+e)W=(xbx + b + (xb + bx))W
\] 
is a sum of  $64$ terms in the set $H$.  Hence,  $(xb+bx)W$ is a sum of $2\cdot 64+1=129$ terms in the set $H$ for all $\|x\|\leq 3$.

For $\|x\| \leq 1$  we have $\|xb-bx\| \leq 2$. Applying the conclusion from the previous paragraph  to 
\[
(xb-bx)b + b(xb-bx)=xb^2-b^2x,
\] 
we get that $(xb^2-b^2x)W$ is a sum of $129$ terms in $H$ for $\|x\| \leq 1$.  Also, applying the same arguments to $b^2$, the element $(xb^2+b^2x)W$ is a sum of 
$129$ unitary conjugates of $\pm b^2W,\pm i b^2W$ by unitaries in $DU_0(M_2(B))$. Further, as observed above,  $b^2$ is a sum of  $64$ elements in $H$, i.e., conjugates of $\pm bW$ and $\pm i bW$ by unitaries in $DU_0(M_2(B))$. Hence, $(xb^2+b^2x)W$ is a sum of $64\cdot 129$ elements in $H$.  Now adding 
$xb^2 + b^2 x$ and $xb^2-b^2x$ we get that $xb^2W$ is a sum of $65\cdot 129=8385$ elements in $H$ 
for $\|x\|\leq 1$.

Let $u\in U(B)$. (Here we use the convention that  if $B$   is non-unital the unitaries in $U(B)$ are chosen in $B^\sim$ and of the form $1+v$, with $v\in B$.) We have that
\[
\begin{pmatrix} 0 & ucu\\ 0 & 0\end{pmatrix}=
\begin{pmatrix} u & 0  \\ 0 & u^*\end{pmatrix}
\begin{pmatrix} 0 & c  \\ 0 & 0\end{pmatrix}
\begin{pmatrix} u^* & 0  \\ 0 & u\end{pmatrix}.
\]
The unitary $\begin{pmatrix} u & 0  \\ 0 & u^*\end{pmatrix}$ belongs to $DU_0(M_2(B))$. (Proof: It suffices to assume that $\|u-1\|<1$, as any unitary is a product of such $u$. In this case the diagonal unitary is a single commutator by \cite[Lemma 5.17]{delaHarpeSkandalis2}.) Thus, $(ucu)W$ is a unitary conjugate of $cW$ by a unitary in $DU_0(M_2(B))$.  Hence, $uxb^2uW$ is a sum of $8385$ elements in $H$ (the same number as $xb^2W$) for any $u\in U(B^\sim)$. Since $x$ is arbitrary such that $\|x\|\leq 1$, replacing $x$ by $u^*x$ we get that
$(xb^2u)W$ is a sum of $8385$ elements in $H$ for any $\|x\|\leq 1$ and $u\in U(B)$. But every element in $B$ of norm $\leq 1$ is a sum of four unitaries in $U(B)$. 
Hence, $(xb^2y)W$ is a sum of $4\cdot 8385=33540$ elements in $H$ for $\|x\|, \|y\|\leq 1$ , equivalently, for $\|x\| \cdot \|y\|\leq 1$.
\end{proof}

The following estimates around the exponential function are easy to derive. We include their proof for the convenience of the reader.
\begin{lemma}\label{expXY}
Let $0<|t|<1$, $x,y\in A$, and $v\in A$.
\begin{enumerate}[(i)]
\item
We have that
\[
[v,x] = \frac{1}{t}(e^{tx}ve^{-tx} - v) + t\Delta_1.
\]
where $\|\Delta_1\|\leq \|v\|\cdot e^{2\|x\|}$.
\item
We have that
\[
[[v,x],y]  = \frac{1}{t^2}(e^{ty}\Big( e^{tx}ve^{-tx}-v\Big)e^{-ty}) - \frac{1}{t}(e^{tx}ve^{-tx}-v) + t\Delta_2
\]
where 
\[
\|\Delta_2\| \leq \|v\|\cdot (e^{2\|x\|+2\|y\|}+2\|y\|e^{2\|x\|})
\]
\end{enumerate} 
\end{lemma}

\begin{proof}
(i) Recall that 
\[
e^{x}ve^{-x}=v+[v,x]+\frac{1}{2!}[[v,x],x]+\frac{1}{3!}[[[v,x],x],x]]+\cdots
\]
(Put differently, $\mathrm{Ad}_{e^x}=e^{\mathrm{ad}_x}$; see \cite[Proposition 5.16]{hofmann-morris}). Replacing $x$ by $tx$ we get 
\[
e^{tx}ve^{-tx} = v + t[v,x] + t^2\Delta_1,
\]
where 
\[
\|\Delta_1\| \leq \|v\| \frac{e^{2\|x\| |t|}-1 -2|t|\|x\|}{t^2}\leq \|v\|e^{2\|x\|}.
\]
Isolating $[v,x]$ the formula of the lemma readily follows.

(ii) Using (i), we have
\begin{align*}
[[v,x],y] &= [\frac{1}{t}(e^{tx}ve^{-tx}-v) + t\Delta_1,y]\\
            & =  [\frac{1}{t}(e^{tx}ve^{-tx}-v),y] +t[\Delta_1,y]
\end{align*}
Using (i) again on the first term on the right side we get           
\[
[[v,x],y]=\frac{1}{t}\left(e^{ty}\Big( \frac{1}{t}(e^{tx}ve^{-tx}-v)\Big)e^{-ty} - \frac{1}{t}(e^{tx}ve^{-tx}-v)\right) + t\widetilde\Delta_1
            +t[\Delta_1,y].
\]
where
\[
\|\widetilde\Delta_1\|\leq \Big\|\frac{1}{it}(e^{tx}ve^{-tx}-v)\Big\| e^{2\|y\|}\leq \|v\| e^{2\|x\|}e^{2\|y\|}.
\]
Combining this estimate with (i), the estimate for $\Delta_2=\widetilde\Delta_1+[\Delta_1,y]$ readily follows.
\end{proof}

\section{Additive setting}

Let $A$ be a unital C*-algebra. Let
\begin{equation}\label{Zset}
\ZZ=\{[a,b]:a,b\in A_{\sa},\, \|a\|,\|b\|\leq 1\}\subseteq iA_{\sa}.
\end{equation}

Given $n\in \N$ and a set $X\subseteq A$ we denote by $\sum^n X$ the set of sums of $n$
elements of $X$.

In this section we prove the following theorem.

\begin{theorem}\label{thmadditive}
Let $A$ be  a unital C*-algebra containing a full square zero element. Let $V\subseteq iA_{\sa}$ be a fully noncentral set invariant under conjugation by $DU_0(A)$ and multiplication by $-1$. Then there exists $n\in \N$ such that $\ZZ\subseteq \sum^n V$.  
\end{theorem}

For our applications to automatic continuity we will need a finer version of this theorem which says that the function $V\mapsto n$, assigning
to a set $V$ as in the theorem the least $n$ such that $\ZZ\subseteq \sum^n V$, is locally bounded in  a
suitable sense. We state this theorem next. Given sets $X,Y\subseteq A$, we write $X\subseteq_\epsilon Y$ to indicate that $X\subseteq \{x\in A:d(x,Y)\leq \epsilon\}$.

\begin{theorem}\label{thmadditivemain}
Let $A$ and  $V\subseteq iA_{\sa}$ be as in Theorem \ref{thmadditive}. Then there exist $n\in \N$, a finite set $F\subseteq V$, and $\epsilon>0$, such that 
if $V'\subseteq iA_{\sa}$ is a fully noncentral set invariant under conjugation by $DU_0(A)$ and multiplication by $-1$, and $F\subseteq_\epsilon V'$, then $\ZZ\subseteq \sum^n V'$.
\end{theorem}

The proofs of the preceding theorems will follow after a series of lemmas.

\begin{lemma}\label{Mconstantlemma}
Let $A$ and $V\subseteq iA_{\sa}$ be as in Theorem \ref{thmadditive}.  Then there exists $M\in \N$ such that 
\begin{equation}\label{Mequation}
1=\sum_{j=1}^M x_j v_j y_j,
\end{equation}
where $x_j,y_j\in A$,  $\|x_j\|, \|y_j\|<1$,  and $v_j\in \pi_3(V^8)$ for all $j=1,\ldots,M$. (Here $\pi_3$ is the nc-polynomial defined in \eqref{ncpolys}.)
\end{lemma}

\begin{proof}
Since $V$ is fully noncentral, $\pi_3(V^8)$ is a  full set, by Lemma \ref{fullpi3}. Therefore, the unit $1\in A$ is  expressible as a finite sum of terms of the form $xvy$, with $x,y\in A$ and $v\in \pi_3(V^8)$. Enlarging the number of terms if necessary we may
assume that $\|x\|,\|y\|<1$ for each of these terms. The lemma is thus proved.  
\end{proof}

For the remainder of this section we assume that $A$ is a unital C*-algebra containing a full square zero element and that $V$ is a fully noncentral  subset of $iA_{\sa}$ invariant under conjugation by $DU_0(A)$ and multiplication by $-1$. We fix $M\in \N$  associated to $V$ through equation \eqref{Mequation}. 

In the next three lemmas we further assume  that
\begin{equation}\label{Vball}
V\subseteq \{x\in A:\|x\|<1\}. 
\end{equation}
We drop this assumption in the proof of Theorem \ref{thmadditivemain}.

\begin{lemma}\label{sumab}
Any  
$z\in \ZZ$ is expressible as a sum of terms of the forms 
\begin{enumerate}[(a)]
\item
$[v,x]$, with $v\in V$ and $x\in \ZZ$,
\item
$[[v,y],z]$, with $v\in V$ and $y,z\in \ZZ$, 
\end{enumerate}
Moreover, the number of terms
of the form (a) is $L_1M$, and the number of terms of the form (b) is $L_2M$,
where  $L_1,L_2\in \N$  are  universal constants. 
\end{lemma}
\begin{proof}
In equation \eqref{Mequation}  multiply by $c\in A_{\sa}$ of norm $\leq 1$ on the left, and take commutator with $d\in A_{\sa}$ also of norm $\leq 1$, to obtain  
\[ 
[c,d]=\sum_{j=1}^M [cx_jv_jy_j,d].
\] 

Fix $1\leq j\leq M$. Say $v_j=\pi_3(v_{j,1},\ldots,v_{j,8})$, with $v_{j,k}\in V$ for $k=1,\ldots,8$.
By Lemma \ref{robnormal}, $[(cx_j)\pi_3(v_{j,1},\ldots,v_{j,8})y_j,d]$ expands into a sum of 
 terms of the forms $[v_{j,k},[r_1,r_2]]$  and  
$[[v_{j,k},[s_1,s_2]],[t_1,t_2]]$, where $r_1, r_2,s_1,s_2,t_1,t_2$ are polynomials on $cx_j,v_{j,1},\ldots,v_{j,8}, y_j$. Since $cx_j$, $v_{j,k}$, and $y_j$ all belong to the open unit ball, the norms of $r_1, r_2,s_1,s_2,t_1,t_2$ are bounded 
by a universal constant. Let $k_1\in \N$ be the number of terms of the form $[v_{j,k},[r_1,r_2]]$
in the expansion of $[(cx_j)\pi_3(v_{j,1},\ldots,v_{j,8})y_j,d]$ from Lemma \ref{robnormal},  and let $k_2\in \N$ be the number of terms of the form  $[[v_{j,k},[s_1,s_2]],[t_1,t_2]]$.

Consider one term of the form $[v_{j,k},[r_1,r_2]]$.
Let us decompose both  $r_1$ and $r_2$ into the sum of  a selfadjoint and a skewadjoint element. The  skewadjoint part of $[v_{j,k},[r_1,r_2]]$ 
is thus expressed as the sum of 2 terms of the form $[v_{j,k},[r_1',r_2']]$, with $r_1',r_2'\in A_{\sa}$. Since the norms of $r_1'$ and $r_2'$
are bounded by a universal constant, we can write $[v_{j,k},[r_1',r_2']]=N_1[v_{j,k},x]$, where $x\in \ZZ$ and $N_1\in \N$ a universal natural number.

We handle similarly 
$[[v_{j,k},[s_1,s_2]],[t_1,t_2]]$: First express  its skewadjoint part as a sum of 8 terms of the form $[[v_{j,k},[s_1',s_2']],[t_1',t_2']]$, where $s_1',s_2',t_1',t_2'\in A_{\sa}$, then choose a universal $N_2\in \N$ such that $[[v_{j,k},[s_1',s_2']],[t_1',t_2']]=N_2[[v_{j,k},y],z]$
for $y,z\in \ZZ$.
Applying this decomposition  across all terms $[cx_jv_jy_j,d]$, we express $[c,d]$ as a sum of obtain  $2k_1N_1M$ terms of the form (a) and  $8k_2N_2M$ terms of the form (b).  
\end{proof}

\begin{lemma}\label{approxV}
Let $N\in \N$. There exists $n\in \N$ such that for each  $z\in \ZZ$ we have that 
\[
\Big\|z-\sum_{j=1}^n v_j\Big\|<\frac1N
\]
for some  $v_1,\ldots,v_n\in V$. Moreover, $n\leq CN^2M^3$, where  $C$ is universal constant.
\end{lemma}

\begin{proof}
We start by writing $z$ as a sum of $L_1M+L_2M$  terms of the forms (a) and (b) of Lemma \ref{sumab}. 

Let us consider a  term of the form  $[v,x]$, with $v\in V$ and $x\in \ZZ$. By Lemma \ref{expXY} (i), 
\[
[v,x] = \frac{1}{t}(e^{tx}ve^{-tx} - v) + t\Delta_1,
\]
where $\|\Delta_1\|\leq \|v\|e^{2\|x\|}$. Observe that $\frac1t(e^{tx}ve^{-tx} - v)$ is the sum of two elements of $\frac1tV$, since $e^{tx}\in DU_0(A)$ (by \cite[Theorem 6.2]{RobertNormal}) and  $V$ is invariant under conjugation by $DU_0(A)$ and multiplication by $-1$.
Since  $\|v\|<1$ (as we have assume that $V$ is contained in the open unit ball) 
and  $\|x\|\leq 2$, $\|\Delta_1\|$ is bounded by $C_1=e^{4}$.  

Let us consider  a  term of the form $[[v,y],z]$, with $v\in V$ and $y,z\in \ZZ$. By Lemma \ref{expXY} (ii), 
\[
[[v,y],z]  = \frac{1}{t^2}e^{tz}\Big( e^{ty}ve^{-ty}-v\Big)e^{-tz} - \frac{1}{t}(e^{ty}ve^{-ty}-v) + t\Delta_2
\]
where 
\[
\|\Delta_2\| \leq \|v\|\cdot (e^{2\|y\|+2\|z\|}+2\|z\|e^{2\|y\|}).
\]
Again, since $e^{ty}$ and $e^{tz}$ belong to $DU_0(A)$,
we have expressed $[[v,y],z]$ as a sum of two elements in $\frac 1{t^2}V$, two elements in $\frac1tV$,  plus the error term $t\Delta_2$.
Since $\|v\|< 1$ and  $\|y\|,\|z\|\leq 2$, we can choose a universal constant $C_2$ bounding  $e^{2\|y\|+2\|z\|}+2\|z\|e^{2\|y\|}$. Then $\|\Delta_2\|\leq C_2$.

Adding all the equations for the terms of the forms $[v,x]$ and $[[v,y],z]$ derived above, we get 
\begin{equation}\label{zapprox}
z=\frac{1}{t}\sum_{j=1}^{2(L_1+L_2)M} v_j+ \frac{1}{t^2}\sum_{j=1}^{2L_2M} v_j' + t\Delta,
\end{equation}
where $v_j,v_j'\in V$ for all $j$ and $\|\Delta\|\leq C_3M$ for  $C_3=\max(C_1,C_2)$. Let $N'\in \N$ be
such that 
\[
C_3MN< N'\leq C_3MN+1.
\] 
Set $t=1/N'$ in \eqref{zapprox}. Then on the right-hand side of \eqref{zapprox} we have a sum of  
\[
2(L_1+L_2)MN'+2L_2M(N')^2
\] 
elements in $V$  plus the error term $\|t\Delta\|<1/N$. Since $N'\leq C_3MN+1$, the number of terms in $V$ is bounded by   $CM^3N^2$, for a suitable universal constant $C$. The lemma is thus proved.
\end{proof}

Next, we get all ``compactly supported'' square zero elements in $\sum^n (V+iV)$ for large enough $n$. Let $x\in A$ be a square zero element. Write $x=|x^*|W$, where $W$ is a partial isometry in the bidual $A^{**}$. (That is, start with $x^*=W^*|x^*|$, the polar decomposition of $x^*$ in $A^{**}$, and take adjoints.) Let $f\in C_c(0,\|x\|]$, where $C_c(0,\|x\|]$ denotes the  functions  of compact support in $C_0(0,\|x\|]$. Define $x_f=f(|x^*|)W$, which is an element of $xAx$ and so a square zero element in $A$. Now define
\[
\Ntwo^{1,c}=\{x_f:x\in \Ntwo\hbox{ and }f\in C_c(0,1]\hbox{ such that $\|x\|\leq 1$ and  $0\leq f\leq 1$} \}.
\]

\begin{lemma}\label{getNtwoc}
There exists $n\in \N$ such that $\Ntwo^{1,c}\subseteq \sum^n (V+iV)$. Moreover, $n\leq C'M^3$,  where $C'$ is a universal constant. 
\end{lemma}

\begin{proof}
Let $x\in \Ntwo$, with $\|x\|\leq 1$.  Write $x=bW$, with $b=|x^*|$ and $W$ a partial isometry in $A^{**}$. 
Let $f\in C_c(0,1]$ be such that $0\leq f\leq 1$, and consider  the element $x_f=f(b)W$ in $\Ntwo^{1,c}$. 
We show below that $x_f\in \sum^n (V+iV)$, with $n$ independent of $x$ and as in the statement of the lemma.

Choose $g\in C_c(0,1]$ such that $0\leq g\leq 1$ and $f=(g^2-\frac 12)_+$.
(To define $g$, let $\delta>0$ be such that $f|_{(0,\delta]}=0$. Now set $g(t)=\sqrt{f(t)+1/2}$ for $t\geq \delta$, $g(t)=0$ for $t<\delta/2$, and let $g$ be linear on the interval $[\delta/2,\delta]$.)

Any  $y\in \Ntwo$ can be expressed as a commutator by writing $y=[w|y|^{\frac12},|y|^{\frac12}]$, where $y=w|y|$ is the polar decomposition of $y$ in the bidual $A^{**}$. If $\|y\|\leq 1$, then $y_1=w|y|^{\frac12}$ and $y_2=|y|^{\frac12}$ also have norm $\leq 1$. Decomposing $y_1$ into its selfadjoint and skewadjoint parts, we deduce that $y\in \ZZ+i\ZZ$ for all $y\in \Ntwo$
of norm $\leq 1$.

Consider $\frac12g(b)W$, which is a square zero element in $A$. As argued in the previous paragraph, there exist $z_1,z_2\in \ZZ$ such that $\frac12g(b)W=z_1+iz_2$. Applying Lemma \ref{approxV} with $N=20$ to $z_1$ and $z_2$ we obtain $n\in \N$ (depending only on $V$) and  $v\in \sum^n (V+iV)$  such that  
\[
\|\frac12g(b)W-v\|<\frac 1{10}.
\]
Moreover, $n\leq 400CM^3$, where $C$ is the universal constant from Lemma \ref{approxV}.

Choose positive contractions $e,e'\in C^*(b)_+$  such that $eg(b)=g(b)$ and $e'e=e$. (They are guaranteed to exist since $g$ vanishes on a neighborhood of 0.) Set $h=W^*eW$, which is an element of $\her(x^*x)$.  (The mapping $a\mapsto W^*aW$ is a C*-algebra isomorphism from $\her(xx^*)$ to $\her(x^*x)$.) Multiplying $\frac12g(b)W-v$ on the left by $e$ and on the right by $h$ we get that
\[
\|\frac12g(b)W-e v h\|<\frac1{10}.
\] 
Set $v'=evh$.  Then $v'=yy^*vy^*y$, where $y=e^{\frac12}W\in A$ is a square zero element of norm at most $1$. 
It then follows from Lemma \ref{lemmaxhx},
applied twice, that $v'$ is a sum of $64$ conjugates of $\pm v$ and  $\pm iv$ by unitaries in $DU_0(A)$. Hence,
$v'\in \sum^{64n} (V+iV)$.

Observe that $v'=cW$, where $c=ev(W^*e)\in A$ is such that $e'c=ce' =c$, and in particular,  $c\in \her(xx^*)$. 
Let us show that $c^*W$ belongs to $\sum^{8n} (V+iV)$. Since $W^*e$ is a square zero element in $A$ of norm $\leq 1$, 
$(W^*e)v(W^*e)$ is a sum of 8 conjugates of $\pm v$ and $\pm i v$ by unitaries in $DU_0(A)$, by Lemma \ref{lemmaxhx}. Hence, $(W^*e)v(W^*e)$ 
belongs to $\sum^{8n} (V+iV)$. But
\[
c^*W= (W^*e)^*v^*eW= ((W^*e)v(W^*e))^*.
\] 
Since $V+iV$ is a selfadjoint set, $c^*W$ is an element of $\sum^{8n} (V+iV)$.

We have that
\[
\|\frac12g(b)-c\|=\|\frac12g(b)W-cW\|=\|\frac12g(b)W-v'\|<\frac1{10}. 
\]
Set $c'=c+c^*$. Then $\|g(b)-c'\|<1/5$. From $\|g(b)-c'\|<1/5$ and $\|g(b)\|\leq 1$ we easily deduce that  $\|(g(b)^2-(c')^2\|<\frac 12$. Hence, by the Kirchberg-R{\o}rdam Lemma (\cite[Lemma 2.2]{KirchbergRordam2}), there exists a contraction $d\in B$ such that
\[
f(b)=(g(b))^2-\frac12)_+=d(c')^2d^*.
\]
We now apply Lemma \ref{tricky} in the C*-algebra $D:=\her(x^*x+xx^*)\cong M_2(\her(xx^*))$ to get the square zero element $x_f=f(b)W$ expressed as a  sum of 33540  unitary conjugates of $c'W$, where the unitaries are in 
$DU_0(D+\C 1_A)\subseteq DU_0(A)$. On the other hand, since $cW\in \sum^{64n}(V+iV)$ and $c^*W\in \sum^{8n}(V+iV)$,
we have that 
\[
c'W=(c+c^*)W\in \sum^{72n}(V+iV).
\]  
It follows that $x_f$ is a sum of $33540\cdot 72n$ elements in $V$. Since $n\leq 400CM^3$, the lemma readily follows.
\end{proof}

Recall that we assume throughout this section that $A$ is a unital C*-algebra containing a full square zero element.
These hypotheses come into play in the proof of the following lemma.

\begin{lemma}\label{fullNtwoc}
There exists $m\in \N$ such that $\ZZ\subseteq \sum^{m} \Ntwo^{1,c}$.
\end{lemma}

\begin{proof}
By assumption, there exists  a full element $x_0\in \Ntwo$. Hence, the positive element $x_0x_0^*$ is also full.
We can express this as a relation in the Cuntz semigroup of $A$:  there exists $k\in \N$ such that 
$[1]_{\Cu}\leq k[x_0x_0^*]_{\Cu}$. Here $[a]_{\Cu}$ denotes the Cuntz class of a positive element $a\in A$ in the Cuntz semigroup $\mathrm{Cu}(A)$. 

Consider the set 
\[
\mathcal Y=\{x\in \Ntwo:[1]\leq k[xx^*]_{\Cu}\}.
\] 
Observe that $\mathcal Y$ is non-empty as $x_0\in \mathcal Y$. 

\emph{Claim}: There exists $n_0\in \N$ such that  each $z\in \ZZ$ is a sum of $n_0$ elements of $\mathcal Y$ of norm at most 1.
Proof: Let us show first that $\mathcal Y$ is invariant under multiplication by nonzero scalars,
fully noncentral, and  invariant under similarity. Clearly, $\mathcal Y$ is closed under multiplication by non-zero scalars, as the Cuntz class of a positive element does not change under multiplication by a positive scalar.
Moreover,  $\mathcal Y$ is fully noncentral, for if $x_0$ maps into the center in some quotient $A/I$, then it maps to 0, as the only central square zero element is zero. But then $I=A$, by the fullness of $x_0$. Let us prove invariance under similarity.
Let $x\in \Ntwo$  and let $s\in A$ be invertible. Set $y=sxs^{-1}$, which is a square zero element.  Let us show that $yy^*$ and $xx^*$ are Cuntz equivalent positive elements. We have
\[
yy^*=s^{-1}xss^*x^*(s^{-1})*\precsim_{\Cu} xss^*x^*\leq \|s\|^2 xx^*.
\] 
This shows that $yy^*\precsim_{\Cu} xx^*$. Since $x$ is also similar to $y$, $xx^*\precsim_{\Cu} yy^*$. Thus $[xx^*]_{\Cu}=[yy^*]_{\Cu}$, and in particular if $x\in \mathcal Y$ then $y\in \mathcal Y$ as well. We can now apply 
Theorem \ref{additivesim}. It follows from this theorem that there exist $n_0\in\N$ and $C>0$ such that
every $z\in \ZZ$ is a sum of $n_0$ elements in $\mathcal Y$ of norm $\leq C$. Since $\mathcal Y$ is closed under multiplication by nonzero scalars, we can assume that $C=1$, enlarging $n_0$ if necessary. This proves the claim.

 In view of the previous claim, to prove the lemma it suffices to show that there exists $n_1\in \N$ such 
 that each element in $\mathcal Y$ or norm $\leq 1$ is a sum of $n_1$ elements of $\Ntwo^{1,c}$.

Let $x\in \mathcal Y$ be such that $\|x\|\leq 1$. Let $B=\her(xx^*)$ and write $x=bW$, with $b=|x^*|\in B$ and $W\in A^{**}$
a partial isometry.

Since $[1]_{\Cu}\leq k[b]_{\Cu}$,  there exists $\epsilon>0$ such that 
$[1]_{\Cu}\leq k[(b-\epsilon)_+]_{\Cu}$. Then
\[
1=\sum_{j=1}^k x_j(b-\epsilon)_+ x_j^*,
\] 
for some $x_1,\ldots,x_k\in A$.
Multiplying by $b^{\frac12}$ on the left and by $b^{\frac12}W$ on the right we obtain that 
\begin{equation}\label{xsum}
x=bW=\sum_{j=1}^k b^{\frac 12}x_j(b-\epsilon)_+x_j^* b^{\frac 12}W.
\end{equation}

Let $f\in C_c(0,1]$ be such that $0\leq f\leq 1$ and $f(t)=1$ for $t\geq \epsilon$. Set $b'=f(b)$, and observe that  $b'(b-\epsilon)_+=(b-\epsilon)_+$. Then, for the terms on the right-hand side of \eqref{xsum} we have that
\begin{align*}
b^{\frac 12}x_i(b-\epsilon)_+x_j^* b^{\frac 12}W &=(b^{\frac 12}x_j((b-\epsilon)_+)^{\frac 12})(b')^2(((b-\epsilon)_+)^{\frac 12}x_j^* b^{\frac 12})W\\
&= r (b')^2r^*W,
\end{align*}
where $r =b^{\frac 12}x_j((b-\epsilon)_+)^{\frac 12}\in B$.
Observe that $\|r\|\leq 1$, since $\|b\|\leq 1$ and   $x_j(b-\epsilon)_+ x_j^*\leq 1$.
We now apply Lemma \ref{tricky}
in the C*-algebra $D=\her(x^*x+xx^*)\cong M_2(B)$. By this lemma we can express  $r(b')^2r^*W$ as a sum of 33540 conjugates of $\pm b'W$ and $\pm i b'W$ by unitaries in 
$DU_0(D+\C1_A)\subseteq DU_0(A)$. Since  
$b'W\in \Ntwo^{1,c}$, every term on the right-hand side of \eqref{xsum} is a sum of $33540$ elements in $\Ntwo^{1,c}$. Thus,
$x$ is a sum of $33540k$ such terms. The lemma is thus proved.
\end{proof} 

\begin{proof}[Proofs of Theorems \ref{thmadditive} and \ref{thmadditivemain}]
Clearly it suffices to prove Theorem \ref{thmadditivemain}.
Let $V$ be a set as in this theorem. Suppose additionally that $V$ is contained in the open unit ball. 
Combining  Lemmas \ref{getNtwoc} and \ref{fullNtwoc} we conclude that there exists $n\in \N$ such that $\ZZ\subseteq \sum^n(V+iV)$, and  comparing skewadjoint parts we further deduce that $\ZZ\subseteq\sum^n V$.  
Here $n\leq C'mM^3$, where $m$ is as in Lemma \ref{fullNtwoc},  $C'$ is the universal constant from Lemma \ref{getNtwoc}, and $M$ is associated to $V$ as in Lemma \ref{Mconstantlemma}.

Consider the equation
\[
1=\sum_{j=1}^M x_j\pi_3(v_{j,1},\ldots,v_{j,8})y_j,
\]
with $\|x_j\|,\|y_j\|<1$ for all $j$ and $v_{j,k}\in V$ for all $j,k$. Clearly then we can choose $\epsilon>0$ such that
if $\|v_{j,k}-v_{j,k}'\|<\epsilon$ for some elements $v_{j,k}'\in A$ with $j=1,\ldots,M$ and $k=1,\ldots,8$, then 
\[
w=\sum_{j=1}^M x_j\pi_3(v_{j,1}',\ldots,v_{j,8}')y_j
\]
is invertible. Moreover,
we can  decrease  $\epsilon>0$ if necessary so that $\|w^{-1}x_j\|<1$ and $\|v_{j,k}'\|<1$ for all $j,k$. 
Let $F=\{v_{j,k}:j=1,\ldots,M,\, k=1,\ldots,8\}$.
Now let $V'\subseteq iA_{\sa}$ be a fully noncentral set invariant under conjugation by unitaries in $DU_0(A)$
and multiplication by $-1$. Suppose further that  $F\subseteq_\epsilon V'$. Choose $v_{j,k}'\in V'$ such that $\|v_{j,k}-v_{j,k}'\|<\epsilon$ for all $j,k$. Let
\[
W'=\{\pm uv_{j,k}'u^*:j=1,\ldots,M,\, k=1,\ldots,8,\, u\in DU_0(A)\}\subseteq V'.
\]
Then $W'$ is a  fully noncentral subset of $iA_{\sa}$ invariant under conjugation by unitaries in $DU_0(A)$
and multiplication by $-1$, and $W'$ is contained in the open unit ball. The equation
\[
1=\sum_{i=1}^M (w^{-1}x_j)\pi_3(v_{j,1}',\ldots,v_{j,8}')y_j
\]
shows that it shares the constant $M$ from Lemma \ref{Mconstantlemma} with $V$. Hence, 
\[
\ZZ\subseteq \sum^n W'\subseteq \sum^n V'.
\] 
This proves the theorem in the case that $V$ is  contained in the open unit ball.

Finally, suppose that $V$ is as in Theorem \ref{thmadditivemain}, but it is not necessarily contained in the unit ball. 
By Lemma \ref{Mconstantlemma}, 
\[
1=\sum_{i=1}^M x_j\pi_3(v_{j,1},\ldots,v_{j,8})y_i,
\]
for some $M\in \N$, $v_{j,k}\in V$, and  $\|x_j\|,\|y_j\|<1$. Let $L=1+\max_{j,k} \|v_{j,k}\|$. 
Define the set
\[
W=\{\pm uv_{j,k}u^*:j=1,\ldots,M,\, k=1,\ldots,8,\, u\in DU_0(A)\}\subseteq V.
\]
Then $\frac{1}{L}W$ is a fully noncentral subset of $iA_{\sa}$ 
invariant under conjugation by unitaries in $DU_0(A)$ and multiplication by $-1$, and additionally 
$W$ is contained in the open unit ball. Hence, $\ZZ\subseteq \frac{1}{L}\sum^nW$ for some $n\in \N$. Since 
$\frac{1}{L}\ZZ\subseteq \ZZ$, we also have that $\ZZ\subseteq \sum^n W$. Similarly, if a finite set $F\subseteq \frac 1LW$ and $\epsilon>0$ satisfy the requisite  property relative to $\frac1LW$, then $LF\subseteq W$ and $L\epsilon$ have the same property relative to $W$. Thus, the validity of the theorem for $\frac{1}{L}W$ implies that 
the theorem is valid for $W$, whence also for $V$.
\end{proof}

\begin{remark}
If  $A$ is a simple unital C*-algebra other than $\C$, then it contains a non-zero (hence full) square zero element,
by Glimm's halving lemma. If $A$ has no 1-dimensional representations and stable rank one,
then again the existence of a full square zero element is guaranteed by \cite[Theorem 9.1]{APRT}. So 
in these cases we can drop the existence of a full square zero element from the hypotheses of the previous theorem (the fullness of the set $[H,A]$ already precludes the existence of 1-dimensional representations).  If $A$
has no finite dimensional representations and   real rank zero, then again it contains a full square zero element by 
\cite[Corollary 2.4]{elliott-rordam}. It is an open problem, known as the ``Global Glimm halving problem", whether every unital C*-algebra without finite dimensional representations must contain a full square zero element. 
\end{remark}

\section{The multiplicative setting}
Throughout this section $A$ denotes a unital C*-algebra.

\begin{lemma}\label{rouviere}
For all $n\in \N$ there exists $\epsilon>0$ such that if $h_1,\ldots,h_n\in [A,A]\cap A_{\sa}$ are such that
$\|h_k\|<\epsilon$ for  $k=1,\ldots,n$, then
\[
e^{i(\sum_{k=1}^n h_k)}=\prod_{k=1}^n e^{ih_k'},
\]	
where $h_k'$ is a conjugate of $h_k$ by a unitary in $DU_0(A)$ for all $k$.
\end{lemma}	
\begin{proof}
The case $n=2$ is proven in \cite[Theorem 5.2]{RobertNormal}. More specifically, there exists $\epsilon>0$ such that 
if $\|h_1\|,\|h_2\|<\epsilon$ then
\[
e^{i(h_1+h_2)}=e^{ir}e^{ih_1}e^{-ir}e^{is}e^{ih_2}e^{-is},
\]
where $r,s\in A_{\sa}$ are such that $r\in h_1+[A,A]$ and $s\in h_2+[A,A]$. Since $h_1,h_2\in [A,A]$, it follows that $r,s\in [A,A]$, and so $e^{ir},e^{is}\in DU_0(A)$ (\cite[Theorem 6.2]{RobertNormal}).

The general case follows by induction and the $n=2$ case.	
\end{proof}

Given $X\subseteq A$ and $n\in \N$, we denote by $X^n$ the set of products of $n$ elements of $X$. 

Given $a,b\in GL(A)$, we denote their multiplicative commutator  $aba^{-1}b^{-1}$ by $(a,b)$.

For a unitary $u\in U(A)$, $a\in A_{\sa}$, and $0<|t|<\frac{\ln(2)}{2\|a\|}$ define 
\[
W_t(u,a)=\frac1{it}\log((u,e^{ita})).
\]
Extend $W_t(u,a)$ to $t=0$  by continuity  by setting  $W_0(u,a)=uau^*-a$.  We have that $W_t(u,a)\in [A,A]\cap A_{\sa}$
for all $t$ by \cite[Theorem 5.1]{RobertNormal} (essentially by the Campbell-Baker-Hausdorff-Dynkin formula). 

Recall that we call a set $H\subseteq A$ fully noncentral if $[H,A]$ is a full subset of $A$. 
Recall also that we denote by $\ZZ$ the set
$\{[a,b]:a,b\in A_{\sa}, \, \|a\|,\|b\|\leq 1\}$. We denote this set by $\ZZ_A$ if reference to the underlying C*-algebra is necessary. 
A set $H\subseteq GL(A)$ is called symmetric if $H=H^{-1}$.

\begin{theorem}\label{mainmult}
Let $A$ be a unital C*-algebra containing a full square zero element. Let $H\subseteq U_0(A)$ be a fully noncentral symmetric set
invariant under conjugation by $DU_0(A)$. Then there exists $n\in \N$ such that 
$
\{ e^{z}:z\in \ZZ\}\subseteq H^n
$.
\end{theorem}
\begin{proof}
Let $B=C([0,1],A)$. Define $V\subseteq iB_{\sa}$ as the smallest set containing the functions 
\[
\Big\{[0,1]\ni t\mapsto iW_t(u,a):u\in H,\,  a\in [A,A]\cap A_{\sa},\, \|a\|\leq \frac{\ln(2)}{4}\Big\},
\]
and invariant under conjugation by unitaries in $DU_0(B)$ and multiplication by $-1$.
Let us show that  $V$ is fully noncentral. Suppose for the sake of contradiction that $V$ becomes central on a nonzero 
quotient of $B$, and assume without loss of generality that this quotient is simple. Then the quotient map factors through  a point evaluation homomorphism. Thus, there exists
 $t_0\in [0,1]$ such that the image of $V$ under the point evaluation at $t=t_0$ is not fully noncentral. Hence, for a nonzero quotient
 $A/I$, the set 
\[
\Big\{iW_{t_0}(u,a):u\in H,\, a\in [A,A]\cap A_{\sa},\, \|a\|\leq \frac{\ln(2)}{4}\Big\}
\] 
is mapped to the center of $A/I$ by the quotient map. Since $\frac1tW_{t_0}(u,ta)\to uau^*-a$
as $t\to 0$, 
$uau^*-a$ is mapped to the center of $A/I$ for all $u\in H$ and $a\in [A,A]\cap A_{\sa}$.  Let $\bar u,\bar a$ denote the images
of $u\in H$ and $a\in [A,A]\cap A_{\sa}$ in $A/I$.  Since $\bar u$ commutes with $\bar u\bar a\bar u^*-\bar a$, $\bar u\bar a\bar u^*-\bar a$ is quasinilpotent by the Kleinecke-Shirokov theorem (\cite{kleinecke}). Since 
it is also a selfadjoint element, $\bar u\bar a\bar u^*-\bar a=0$. This readily implies that $\bar u$ commutes with $[A/I,A/I]$; in particular,  $[\bar u,[\bar u,A/I]]=0$ for all $u\in H$. By Herstein's theorem \cite[Theorem 1]{herstein}, $H$ is mapped to the center in $A/I$, which contradicts that  $H$ is fully noncentral.   

The C*-algebra $B$ contains full square zero elements, namely, any constant function $t\mapsto x$, where $x\in A$ is a full  square zero element. Since  $B$ and $V$ satisfy the requisite hypotheses from Theorem \ref{thmadditive}, there exists $n_1\in \N$ such that  for each $z\in \ZZ$ we have that 
\[
z= \sum_{j=1}^{n_1} (-1)^{k_j}iW_t(u_j(t),a_j(t))\quad (0\leq t\leq 1) 
\]
where the left-hand side is regarded as a constant function in $B$, and where $u_j(t)\in H$ and $a_j(t)\in [A,A]\cap A_{\sa}$ is of  norm $\leq \frac{\ln(2)}{4}$ for all $j$ and all $t$. Multiplying by $t$ on both sides and exponentiating we get
\begin{equation}\label{beforerouviere}
e^{tz}=e^{\sum_{j=1}^{n_1} (-1)^{k_j}\log((u_j(t),e^{ita_j(t)}))}.
\end{equation}
Recall that  $\log((u_j(t),e^{ita_j(t)}))\in [A,A]\cap iA_{\sa}$ (e.g., by \cite[Theorem 5.1]{RobertNormal}).
Let $\delta>0$ be such that Lemma \ref{rouviere} is valid for $n_1$  elements of norm $<\delta$.
Let $\epsilon>0$ be such that
\[
\|\log((u,e^{ita}))\|<\delta
\]
for all  $\|a\|\leq \frac{\ln(2)}{4}$, $0\leq t<\epsilon$, and all unitaries $u$. Then, applying Lemma \ref{rouviere} on the right-hand side of \eqref{beforerouviere} we get
\begin{equation}\label{lasteqMult}
e^{tz}=\prod_{j=1}^{n_1} (u_j'(t),e^{ita_j'(t)})^{\pm 1}.
\end{equation}
for all $0\leq t<\epsilon$, where $u_j'(t)$  and $a_j'(t)$ are conjugates of $u_j(t)$ and $a_j(t)$ by unitaries in $DU_0(A)$ for all $j$. 
Since $H$ is invariant under conjugation by $DU_0(A)$,  $u_j'(t)\in H$ for all $j$. Furthermore, since $a_j(t)\in [A,A]\cap A_{\sa}$,
$e^{ita_j'(t)}\in DU_0(A)$ (by \cite[Theorem 6.2]{RobertNormal}). So the right-hand side of \eqref{lasteqMult} 
belongs to $H^{2n_1}$.  Choose $n_2\in \N$ such that $1/n_2<\epsilon$. Then $e^{\frac1{n_2} z}\in H^{2n_1}$ for all $z\in \ZZ$, and so $e^{z}\in H^{2n_2n_1}$ for all  $z\in \ZZ$.
\end{proof}

\begin{proof}[Proof of Theorem \ref{DUsimple}]
If $A=\C$, the theorem is trivial. Assume that $A\neq \C$. Then, by Glimm's lemma, $A$ contains a nonzero (and necessarily full)
square zero element. 
  
Let $G$ be a normal subgroup of $DU_0(A)/Z(DU_0(A))$ not consisting solely of the identity element. Let $H$ denotes the preimage 
of $G$ in $DU_0(A)$. Then $H$ is a fully noncentral normal subgroup of $DU_0(A)$.
By the previous theorem, $H$ contains $e^{z}$, with $z\in \ZZ$. Since these elements generate $DU_0(A) $,  by \cite[Theorem 6.2]{RobertNormal},  $H=DU_0(A)$. It follows that $G=DU_0(A)/Z(DU_0(A))$.
\end{proof}

\begin{theorem}\label{finemult}
Let $A$ be a unital C*-algebra containing a full square zero element. Let $h\in A_{\sa}$ be fully noncentral. For each $\alpha\in \R$ define 
\[
W_\alpha = \{ue^{\pm i\alpha h}u^*:u\in DU_0(A)\}.
\]
Then there exist $n\in \N$ and $\delta>0$ such that for each $0<|\alpha|<\delta$ the set $W_\alpha^n$ contains a set of the form $\{e^{\delta' z}:z\in \ZZ\}$  for some $\delta'>0$.
\end{theorem}

\begin{proof}
Assume without loss of generality that $\|h\|=1$.

The set  
\[
V=\{[uhu^*,a]:u\in DU_0(A),\, a\in [A,A]\cap A_{\sa},\, \|a\|\leq \frac{\ln 2}{4}\}\subseteq iA_{\sa}
\] 
is fully noncentral, invariant under conjugation by $DU_0(A)$ and under multiplication by $-1$.
(If $[h,[A,A]]$ is mapped to the center in a non-zero quotient, then $h$ is also mapped to the center, by Herstein's \cite[Theorem 1]{herstein}, violating that $h$ is fully noncentral.)  Set $B=C([0,1],A)$ and define $\widetilde V\subseteq iB_{\sa}$ by 
\[
\widetilde V=\{\bar uv\bar u^*: \bar u\in DU_0(B),\, v\in V\},\] 
where $V$ is regarded as a set of constant functions in $B$. The set 
$\widetilde V$ is again fully noncentral, as it is mapped onto $V$ by any point evaluation map, and it is  invariant under conjugation by $DU_0(B)$  and under multiplication by $-1$. 
Let $n_1\in \N$, $F\subseteq \widetilde V$, and $\epsilon_1>0$, be  associated to $\widetilde V$ as in Theorem \ref{thmadditivemain}, i.e., such that if $F\subseteq_{\epsilon_1} \widetilde V'$  for a fully noncentral  set  $\widetilde V'\subseteq iB_{\sa}$ invariant under conjugation by $DU_0(B)$ and multiplication by $-1$, 
then $\ZZ_B\subseteq \sum^{n_1} \widetilde V'$. 

For $\alpha\neq 0$, define $\widetilde V_\alpha\subseteq iB_{\sa}$ as the smallest set containing the functions 
\[
\Big\{t\mapsto \frac{i}{\alpha}W_t(e^{i\alpha h},a): a\in [A,A]\cap A_{\sa}, \|a\|\leq \frac{\ln 2}{4}\Big\}
\] 
and invariant under conjugation by unitaries in $DU_0(B)$ and multiplication by $-1$. 
Then $\widetilde V_\alpha$ is symmetric, invariant under conjugation by $DU_0(B)$, and fully noncentral.  
Applying the Campbell-Baker-Hausdorff-Dynkin formula three times, we have that for small enough $\alpha$
\begin{align*}
\frac i\alpha W_t(e^{i\alpha h},a) 
&=\frac{1}{\alpha t}\log(e^{i\alpha h}e^{iat}e^{-i\alpha h}e^{-iat})\\
&=\frac{1}{\alpha t}\log(e^{i\alpha h+ iat + \frac{1}{2}[i\alpha h,iat]+\cdots}\cdot e^{-i\alpha h-iat + \frac{1}{2}[i\alpha h,iat]+\cdots})\\
& =[a,h] + \Delta,
\end{align*}
where $\|\Delta\|\leq C|\alpha t|$, for a constant $C$.  (Recall that we have assumed that $\|h\|=1$.)
Thus, for small enough $\alpha$, $\widetilde V_\alpha$ almost contains any constant function of the form $t\mapsto [a,h]$, 
with $\|a\|\leq \frac{\ln(2)}{4}$, as well as any conjugation of one such function by a unitary in $DU_0(B)$. Hence, with $\epsilon>0$ and $F\subseteq \widetilde V$ as in the previous paragraph,  we have that $F\subseteq_\epsilon \widetilde V_\alpha$ for small enough $\alpha$. 
Let $\delta>0$ be such that this holds for $0<|\alpha|<\delta$, and fix any such $\alpha$. Then
$z\in \sum^{n_1}\widetilde V_\alpha$ for all $z\in \ZZ_A$, where $z$ is understood as a constant function in $B$. 
Thus, 
\[
z = \sum_{j=1}^{n_1} (-1)^{k_j}\frac{i}\alpha W_t(u_j(t)e^{i\alpha h} u_j(t)^*,u_j(t)a_ju_j(t)^*)\quad (0\leq t\leq 1), 
\]
where $u_j\in DU_0(B)$ and $a_j\in [A,A]\cap A_{\sa}$ is such that $\|a_j\|\leq \frac{\ln 2}{4}$ for all $j$. Multiplying by $\alpha t$ on both sides and exponentiating we get
\[
e^{\alpha t z}=e^{\sum_{j=1}^{n_1} (-1)^{k_j}\log ((u_j(t)e^{i\alpha h} u_j(t)^*,u_j(t)e^{ita_j}u_j(t)^*))}.
\]
As in the proof of Theorem \ref{mainmult}, we now find $\delta'>0$ such that Lemma \ref{rouviere} 
can be applied to $n_1$ selfadjoint elements of norm $<\delta'$, and then find  $\epsilon>0$ such that 
\[
\|\log ((u,e^{ita})\|<\delta',
\]
for all $0<t\leq \epsilon$, unitary $u$, and $\|a\|\leq \frac{\ln(2)}{4}$. 
Then\[
e^{\alpha tz}=\prod_{j=1}^{n_1} (u_j'(t),e^{ita_j'(t)})^{\pm 1}.
\]
for $0<t\leq \epsilon$, where  $u_j'(t)$ is a conjugate of $e^{i\alpha h}$ by a unitary in $DU_0(A)$ for all $j$ and 
all $0<t\leq \epsilon$. From $a_j'(t)\in [A,A]\cap A_{\sa}$ we get that  $e^{ita_j'(t)}\in DU_0(A)$ for all $j$. 
It follows that $W_{\alpha}^{2n_1}$ contains $\{e^{\alpha \epsilon z}:z\in \ZZ\}$. 
\end{proof}

Under the hypotheses of the previous theorem, fully noncentral selfadjoint elements are guaranteed to exist. 

\begin{lemma}\label{fullh}
If $x\in A$ is a full square zero element then $[x^*,x]$ is a fully noncentral selfadjojnt element. 
\end{lemma}
\begin{proof}
Let $h=[x^*,x]$, where $x$ is a full square zero element of $A$. If $h+I$ is central in a quotient $A/I$, then in particular $x+I$ commutes with $[(x+I)^*,x+I]$. By the Kleinecke-Shirokov theorem \cite{kleinecke},  $[(x+I)^*,x+I]$ is quasinilpotent, which in turn entails that it is zero, since it is also selfadjoint. Hence, $[x^*,x]\in I$. By functional calculus, $x^*x=([x^*,x])_+\in I$, and so $I=A$
by the fullness of $x$.  
\end{proof}

\section{The special unitary group}
Let us briefly recall the definition of the de la Harpe-Skandalis determinant. We refer the reader to \cite{delaHarpeSkandalis} for further details. 

Let $A$ be a unital C*-algebra. Let $M_\infty(A)$ denote the *-algebra of infinite matrices with entries in $A$ whose entries are all but finitely many equal to zero. Regard $M_n(A)$ embedded in $M_\infty(A)$ as the $n\times n$ top left corner matrices, so that 
$M_\infty(A)=\bigcup_{n=1}^\infty M_n(A)$. Endow $M_\infty(A)$ with the inductive limit locally convex topology and 
its unitization $M_\infty(A)+\C1_\infty$ with the product of the topologies of $M_\infty(A)$ and $\C$.
Let $GL_0^\infty(A)\subseteq M_\infty(A)+\C1$ denote the connected component 
of the identity in the group of invertible elements. Regard $GL_0^n(A)$ (i.e., $GL_0(M_n(A))$) 
embedded in 
$GL_0^\infty(A)$ by the map $a\mapsto a+(1_\infty - 1_n)$, where $1_n$ is the identity in $M_n(A)$.
 We endow $GL_0^\infty(A)$ with the topology that it inherits from its inclusion in $M_\infty(A)+\C1_\infty$. In this topology, a compact set is always contained in $M_n(A)+\C1_\infty$ for some $n$. In particular, a path  $\alpha\colon [t_1,t_2]\to GL_0^\infty(A)$ is always contained in $GL_0^n(A)$ for some $n$. Let $U_0^\infty(A)$ denote the unitary elements in $GL_0^\infty(A)$.

Let $T\colon A\to A/\overline{[A,A]}$ denote the quotient map, where the codomain is regarded as a Banach space under the quotient norm. We call $T$ the universal trace map.  
Define $Tr\colon M_\infty(A)\to A$ as addition along the main diagonal: $Tr((a_{ij})_{i,j})=\sum_i a_{ii}$. 
We extend $T$ to $M_\infty(A)$ by setting $T((a_{i,j})_{i,j})=T(Tr((a_{i,j})_{i,j}))$ for all $(a_{i,j})_{i,j}\in M_\infty(A)$.

Let $\alpha\colon [t_1,t_2]\to GL_0^\infty(A)$ be a smooth path. Its de la Harpe-Skandalis determinant is defined as 
\[
\widetilde \Delta_T(\alpha)=\frac{1}{2\pi i}\int_{t_1}^{t_2}T(\alpha'(t)\alpha(t)^{-1})dt\in A/\overline{[A,A]}.
\]
By the results of \cite{delaHarpeSkandalis}, if two paths $\alpha_1$ and $\alpha_2$  both connect $1$ to an invertible element $a\in GL_0^\infty(A)$,
then 
\[
\widetilde\Delta_T(\alpha_1) - \widetilde\Delta_T(\alpha_2)\in 2\pi i\{T(p)-T(q):p,q\in M_\infty(A)\hbox{ projections}\}.
\] 
Given an invertible element $a\in GL_0^\infty(A)$, its de la Harpe-Skandalis determinant  $\Delta_T(a)$ is defined as the image of $\widetilde\Delta_T(\alpha)$ in the quotient 
\[
(A/\overline{[A,A]})\,/\,(2\pi i\{T(p)-T(q):p,q\in M_\infty(A)\hbox{ projections}\}),
\] 
where $\alpha$ is any path in $GL_0^\infty(A)$ connecting $1$ to $a$.  

We are interested in $\ker \Delta_T\cap U_0(A)$, which is a normal subgroup of $U_0(A)$. It can be described as follows: Let $u\in U_0(A)$, and let $u=\prod_{j=1}^n e^{ih_j}$, with $h_j\in A_{\sa}$ for all $j$, be any representation of $u$ as a product of exponentials. Then $u\in \ker \Delta_T$ if and only if
\begin{equation}\label{DTker}
\sum_{j=1} h_j\in \overline{[A,A]}+2\pi\{Tr(p)-Tr(q):p,q\in M_\infty(A)\hbox{ projections}\}.
\end{equation}

Let $\tau\colon A\to \C$ be a tracial state. The de la Harpe-Skandalis determinant $\Delta_\tau$ associated to $\tau$ is defined similarly to $\Delta_T$,
with $\tau$ in place of $T$: Given a smooth path $\alpha\colon [t_1,t_2]\to GL_0^\infty(A)$, 
\[
\widetilde \Delta_\tau(\alpha)=\frac{1}{2\pi i}\int_{t_1}^{t_2} \tau(\alpha'(t)\alpha^{-1}(t))\in \C.
\]    
Given $a\in GL_0^\infty(A)$ we choose any path $\alpha\colon [t_1,t_2]\to GL_0^\infty(A)$ connecting $1$ to $a$ and define $\Delta_\tau(a)$
as the equivalence class of $\widetilde\Delta_\tau(\alpha)$ in the quotient 
\[
\C/2\pi i\{\tau(p)-\tau(q):p,q\in M_\infty(A)\hbox{ projections}\}.
\] 
(Described in terms of the pairing of the group $K_0(A)$ with the tracial states of $A$, we are taking the quotient of $\C$  by the group $2\pi i \tau(K_0(A))$.)
The group $\ker\Delta_\tau\cap U_0(A)$ consists of the unitaries in $U_0(A)$ 
such that in any one (and all) representations as a  
product of exponentials $\prod_{j=1}^n e^{ih_j}$, with $h_j\in A_{\sa}$ for all $j$, one has that  
\begin{equation}\label{Dtauker}
\tau\Big(\sum_{j=1}^n h_j\Big)\in \{\tau(p)-\tau(q):p,q\in M_\infty(A)\hbox{ projections}\}.
\end{equation}

\begin{lemma}\label{SUchar}
Let $u\in U_0(A)$.
The following are equivalent:
\begin{enumerate}[(i)]
\item
There is a path $\eta\colon [t_1,t_2]\to U_0(A)$ starting at 1 and  ending at $u$ such that $\widetilde\Delta_T(\eta)=0$.
\item
$u$ belongs to the subgroup of $U_0(A)$ generated by $\{e^{ih}:h\in \overline{[A,A]}\cap A_{\sa}\}$.
\item
$u=\prod_{k=1}^n e^{ih_k}$, for some $h_1,\ldots,h_n$ such that $\sum_{k=1}^n h_k\in \overline{[A,A]}\cap A_{\sa}$.
\item
There is a path $\eta\colon [0,1]\to U_0(A)$ starting at 1 and  ending at $u$ 
such that $\widetilde\Delta_T(\eta|_{[0,t]})=0$ for all $t\in [0,1]$.	
\end{enumerate}
\end{lemma}

\begin{proof}
(i) $\Rightarrow$ (ii): Choose $t_1=x_0<x_1<\cdots<x_N=t_2$, a sufficiently  fine partition of the interval 
$[t_1,t_2]$ such that $\|u(x_k)-u(x_{k+1})\|<1$ for all $j$. Then 
	$u=\prod_{k=1}^N e^{ih_k}$, where  
$e^{ih_k}=u_{x_{k-1}}^{-1}u_{x_k}$ for all $k$.  By the proof of \cite[Lemma 3]{delaHarpeSkandalis}, 
if the partition $t_1=x_0<x_1<\cdots<x_N=t_2$ 
	is sufficiently fine, then  	$\sum_{k=1}^N h_k\in \overline{[A,A]}$ 
	
(ii)	$\Rightarrow$ (iii) This is obvious.

(iii) $\Rightarrow$ (iv) The path 
$
\eta(t)=\prod e^{ith_k}$, with $t\in [0,1]$, is such that 
\[
\widetilde\Delta_T(\eta|_{[0,t]})=T(t(\sum_{k=1}^n ih_k))=0
\]
for all $t\in [0,1]$.

(iv) $\Rightarrow$ (i) This is obvious.
\end{proof}

Let $SU_0(A)$ denote the subgroup of $U_0(A)$ of unitaries satisfying any of the equivalent conditions of Lemma \ref{SUchar}.
Clearly, $SU_0(A)$ is a path connected normal subgroup of $U_0(A)$
contained  in 	$\ker\Delta_T\cap U_0(A)$. By Lemma \ref{SUchar} (ii), 
$SU_0(A)$ is the Banach-Lie subgroup of $U(A)$ associated to the Lie algebra of skewadjoint commutators 
$\overline{[A,A]}\cap iA_{\sa}$.  We refer the reader to \cite[Chapter 5]{hofmann-morris} for background on Banach-Lie groups. 

We regard $SU_0(A)$ endowed with the topology with basis of neighborhoods at the identity 
\[
\{e^{ih}:h\in \overline{[A,A]}\cap A_{\sa}, \, \|h\|<\epsilon\} \quad  (\epsilon>0).
\]
Then $\overline{[A,A]}\ni h\mapsto e^{ih}\in SU_0(A)$ is continuous (\cite[Theorem 5.52]{hofmann-morris}). 
The exponential length on $SU_0(A)$ is defined as
\begin{equation}\label{elSU0}
\el_{SU_0(A)}(u)=\Big\{\inf \sum_{j=1}^n \|h_j\|:u=\prod_{j=1}^n e^{ih_j}\hbox{ and }h_j\in \overline{[A,A]}\cap A_{\sa}\hbox{ for all }j\Big\}.
\end{equation}
This gives rise to a metric $(u,v)\mapsto \el(u^*v)$ that is left and right invariant and induces the topology on $SU_0(A)$ (see \cite[Proposition 3.2]{ando}). 
It is easily established that $SU_0(A)$ is complete under this metric. 
Moreover, by \cite[Proposition 3.9]{ando}, 
the metric induced by $\el_{SU_0(A)}$  agrees with the ``Finsler metric'' defined by
\[
d_{SU_0(A)}(u,1)=\inf \Big\{\int_0^1\|u'(t)\|dt:t\mapsto u(t)
\hbox{ smooth path connecting 1 to $u$ in }SU_0(A).\Big\}
\]

\begin{remark}\label{notthesametop}
The topology  on $SU_0(A)$ need not be that induced by the topology on  $U_0(A)$. Take for example 
$A$ to be an infinite dimensional simple C*-algebra of real rank zero with at least one tracial state
(e.g. the CAR C*-algebra $M_{2^\infty}$). Then $SU_0(A)$ is a norm dense proper subgroup of $U_0(A)$
(\cite{elliott-rordam0}). In particular, $SU_0(A)$ is not complete under the uniform structure induced  
by the norm topology.
\end{remark}

We summarize various properties of $SU_0(A)$ in the following theorem:

\begin{theorem}\label{SUprops}
Let $A$ be a unital C*-algebra. The following are true:
\begin{enumerate}[(i)]
\item
$DU_0(A)$ is a dense subgroup of $SU_0(A)$. If $\overline{[A,A]}=[A,A]$, then $DU_0(A)=SU_0(A)$.

\item
If $A$ is traceless, then $DU_0(A)=SU_0(A)=U_0(A)$, and the topology on $SU_0(A)$ is that induced by the norm on $A$.

\item
If $A$ is separable, then $SU_0(A)$ agrees with the path connected component of 1 in $\ker\Delta_T\cap U_0(A)$.

\item
If $A$ has stable rank one, then $SU_0(A)=\ker\Delta_T\cap U_0(A)$. 	

\item
If $A$ has real rank zero, then $SU_0(A)=\ker\Delta_T\cap U_0(A)$. 	

\item
If $A$ is simple, then $SU_0(A)$ is topologically simple modulo its center.

\end{enumerate}

\end{theorem}

\begin{proof}
(i) This follows at once from the continuity of 
$\overline{[A,A]}\ni h\mapsto e^{ih}\in SU_0(A)$ and the fact that  $DU_0(A)$ is generated by the set $\{e^{ih}:h\in [A,A]\cap A_{\sa}\}$ (\cite[Theorem 6.2]{RobertNormal}).

(ii)  By Pop's theorem \cite[Theorem 1]{Pop}, if  $A$ is a traceless unital C*-algebra then
 $A=[A,A]$. Since the set $\{e^{ih}:h\in [A,A]\cap A_{\sa}\}$ is a generating set 
 of $DU_0(A)$, it follows that $U_0(A)=DU_0(A)$, and a fortiori $SU_0(A)=U_0(A)$.

(iii) \emph{Claim}: If $[t_1,t_2]\ni t\mapsto u(t)$ is a continuous path in $\ker \Delta_T\cap U_0(A)$ such that $\|u(t)-1\|<1$ for all $t$ and $u(t_1)=1$, 
then $\log(u(t))\in \overline{[A,A]}$ for all $t\in [t_1,t_2]$.	Proof:
Let $h(t)=\log(u(t))$ for all $t\in [t_1,t_2]$, which is defined since $\|u(t)-1\|<1$,
and such that $h(t_1)=0$. Let $\tau$ be a tracial state. Then  $t\mapsto \tau(h(t))$ varies continuously, and by the description of $\ker \Delta_T\cap U_0(A)$ (see \eqref{DTker}),  ranges in the set  
\[
2\pi i\{\tau(p)-\tau(q):p,q\in M_\infty(A)\hbox{ projections}\}.
\] 
This set is countable, since $A$ is separable and the value of a trace on a projection is invariant 
under homotopy. Hence,  $t\mapsto \tau(h(t))$ is constant, and so it is 0 as $h(t_1)=0$. This shows $\tau(h(t))=0$ for all tracial states $\tau$,  which in turn implies that $h(t)\in \overline{[A,A]}$ for all $t$. This proves the claim.

In virtue of the claim just established, if $u\in \ker\Delta_T\cap U_0(A)$ is connected to 1 through a path 
$t\mapsto u(t)$ inside $\ker\Delta_T\cap U_0(A)$ and such that $\|u(t)-1\|<1$ for all $t$, 
then the path lies entirely in $SU_0(A)$. We can drop the assumption that $\|u(t)-1\|<1$
by partitioning the path into small enough pieces.

(iv) Let $u\in \ker\Delta_T\cap U_0(A)$, and write $u=\prod_{k=1}^n e^{ih_k}$, where 
$h_k\in A_{\sa}$ for all $k$. Then 
\[
\sum_{k=1}^n h_k=h+2\pi i Tr(p)-2\pi i Tr(q), 
\]
for some $h\in \overline{[A,A]}\cap A_{\sa}$ and  projections $p,q\in M_\infty(A)$.
By \cite[Lemma 6.1]{RobertNormal}, $u$ is equal modulo commutators (in $DU_0(A)$)
to $e^{ih}e^{2\pi i Tr(p)}e^{-2\pi i Tr(q)}$. Since $DU_0(A)\subseteq SU_0(A)$ and 
$e^{ih}\in SU_0(A)$,  it remains to prove that any unitary of the form  
 $v=e^{2\pi iTr(p)}$, with $p\in M_n(A)$  a projection, belongs to  $SU_0(A)$. 
Consider the projection loop $[0,1]\ni t\mapsto e^{2\pi i tp}$ in $M_n(A)$. By Rieffel's theorems 
\cite[Corollary 8.6]{rieffel1} and \cite[Proposition 2.6]{rieffel2}, the loop $[0,1]\ni t\mapsto e^{2\pi i tp}$ is homotopic 
to a loop $\beta\colon [0,1]\to U_0(A)$ (here we have used the stable rank one hypothesis). Now $\eta(t)=e^{2\pi i tTr(p)}\beta^{-1}(t)$
is a path in $U_0(A)$ connecting $1$ to $e^{2\pi iTr(p)}$ and 
\[
\widetilde \Delta_T(\eta)=\Delta (t\mapsto e^{2\pi i tTr(p)})- \Delta(\beta)=2\pi iTr(p)-2\pi iTr(p)=0.
\]
It follows by Lemma \ref{SUchar} that $u\in SU_0(A)$, as desired.	

(v) Let $u\in \ker\Delta_T\cap U_0(A)$. By Lin's \cite[Theorem 5]{LinRR0}, $u=e^{ih_1}e^{ih_2}$, 
where $\|h_1\|\leq \pi$ and the norm of $h_2$ can be made arbitrarily small. Choosing $h_2$
of a sufficiently small norm we have that $e^{ih_1}e^{ih_2}=e^{i(h_1+h_2)}e^{ic}$, where 
 $c\in \overline{[A,A]}$ (\cite[Lemma 2.2]{ker-det}). Since   $e^{ic}\in SU_0(A)$, 
 it remains to show that if $e^{ih}\in \ker\Delta_T\cap U_0(A)$, 
 then $e^{ih}\in SU_0(A)$. We have that  $h=h'+2\pi (Tr(p)-Tr(q))$,
with $h'\in \overline{[A,A]}\cap A_{\sa}$ and $p,q\in M_\infty(A)$ projections.
As in the proof of (iv), $e^{ih}$ is equivalent modulo commutators to 
$e^{h'}e^{2\pi i Tr(p)}e^{-2\pi Tr(q)}$, so it remains to show that 
$e^{2\pi i Tr(p)}$ belongs to $SU_0(A)$ for any projection $p\in M_n(A)$. 
By Zhang's \cite[Theorem 1.1]{zhang}, the Murray-von Neumann monoid of projections of a C*-algebra of real rank zero
has the Riesz decomposition property. Hence, there exists $v\in M_\infty(A)$
such that $p=v^*v$  and $q=vv^*$ is a diagonal matrix with projections 
$q_1,\ldots,q_n\in A$ along the main diagonal.  Observe that $Tr(v^*v)-Tr(vv^*)\in [A,A]$. 
So $e^{2\pi i Tr(q)}$ is equivalent modulo $DU_0(A)$ to $e^{2\pi i (\sum_{k=1}^n q_k)}$,
which in turn is equivalent modulo $DU_0(A)$ to $\prod_{k=1}^n e^{2\pi i q_k}=1$ 
(\cite[Lemma 6.1]{RobertNormal}), thus completing the proof.

(vi) Let $H\subseteq SU_0(A)$ be a closed normal subgroup not contained in the center of $SU_0(A)$ (thus, not contained in the center of $A$
either). Then $DU_0(A)\subseteq H$, by Corollary \ref{DUsimple}. Since $DU_0(A)$ is dense  in $SU_0(A)$ by part (i), and $H$ is closed, $H=SU_0(A)$.
\end{proof}

\begin{remark}
Consider the inclusions
\[
DU_0(A)\subseteq SU_0(A)\subseteq \ker\Delta_T\cap U_0(A).
\]
Both inclusions may be proper, though it is often the case that all three sets agree (e.g., by a combination of the results in the preceding theorem). 
If $A=M_n(C(X))$, for $X$ a compact Hausdorff space,  then $DU_0(A)=SU_0(A)$  by \cite[Proposition 1.3]{thomsen}, and it is not difficult to show that $SU_0(A) = \ker\Delta_T\cap U_0(A)$.
The examples that we construct in Section \ref{seccounterexamples} below show failure of the first inclusion
for some simple AH C*-algebras.
 
 The question whether $SU_0(A) = \ker\Delta_T\cap U_0(A)$---equivalently, in the separable case,  whether $\ker\Delta_T\cap U_0(A)$
is a path connected subset of $U_0(A)$---is closely related to
whether the C*-algebra $A$ has determinant unitary rank equal to 1, in the sense of Gong, Lin, and Xue \cite{DUR}. 
In \cite[Theorem 5.13]{DUR}, an example is given  of a homogeneous C*-algebra with determinant unitary rank $>1$.
This is also an example where $SU_0(A)\neq \ker\Delta_T\cap U_0(A)$.  Briefly described, $A=pM_4(C(S^4))p$, where $S^4$ denotes the 4-dimensional sphere and 
   $p\in M_4(C(S^4))$ is a suitably chosen rank 2 projection. By the results in \cite{DUR} there exists a projection $q\in M_2(A)$ equivalent to the trivial rank one projection such that  for no loop 
 $\alpha\colon [0,1]\to U_0(A)$ can one have  $\widetilde\Delta_T(\alpha)=2\pi i Tr(q)+\overline{[A,A]}$. It follows that 
 $u=e^{2\pi i Tr(q)}\in U_0(A)$ does not belong to $SU_0(A)$, since a path $\eta$ of zero determinant connecting $1$ to $u$ concatenated  with the path 
 $[0,1]\ni t\mapsto e^{2\pi i (1-t)Tr(q)}$ would provide one such  loop.
 On the other hand, $u$ is clearly in the kernel of the determinant.
\end{remark}

\section{Automatic continuity}
Let us recall the definition of the property of bounded commutators generation given in the introduction.  
Let $A$ be a unital C*-algebra. We say that $A$ has bounded commutators generation if there exist $n\in\N$ and $C>0$ such that for all $h\in \overline{[A,A]}$ we have 
\[
h=\sum_{j=1}^n [x_j,y_j]\] 
for some $x_j,y_j\in A$ such that $\|x_j\|\cdot \|y_j\|\leq C\|h\|$ for all $j$.

In the following theorem we gather some classes of C*-algebras known to have bounded commutators generation:

\begin{theorem}\label{BCGclasses}
The following classes of unital C*-algebras have bounded commutators generation:
\begin{enumerate}[(i)]
\item
traceless C*-algebras,
\item
C*-algebras having finite decomposition rank and without finite dimensional representations, 
\item
pure C*-algebras (i.e., with almost unperforated and almost divisible Cuntz semigroup) 
where every bounded 2-quasitrace is a trace,
\item
exact C*-algebras tensorially absorbing the Jiang-Su C*-algebra.
\end{enumerate}
\end{theorem}
\begin{proof}
(i) This is Pop's \cite[Theorem 1]{Pop}.

(ii) This follows from  \cite[Theorem 1.2]{RobertCommutators}.

(iii) This follows from \cite[Theorem 4.10]{NgRobertpure}.

(iv) This follows from (iii), as a Jiang-Su stable C*-algebra is pure. Moreover, exactness implies that bounded 2-quasitraces are traces, by Haagerup's theorem (\cite{haagerup}). Thus, the conditions in (iii) are met.
\end{proof}

Recall that we denote by $\ZZ$  the subset of $[A,A]\cap iA_{\sa}$ defined in \eqref{Zset}.

\begin{lemma}\label{Zneighborhood}
Let $A$ be a unital C*-algebra containing  a full square zero element and having bounded commutators generation. Then there exists $N\in \N$ such that 
\[
\{e^{\epsilon z}:z\in \ZZ\}^N
\] 
is a neighborhood of the identity in $SU_0(A)$ for all $\epsilon>0$.
\end{lemma}

\begin{proof}
By the property of bounded commutators generation, there exists $N\in \N$ such that if $h\in \overline{[A,A]}\cap A_{\sa}$ is of norm $\leq 1$, then 
\[
ih = \sum_{j=1}^{N}z_j 
\] 
for some  $z_j\in \ZZ$. Let $\epsilon_0>0$ be as in Lemma \ref{rouviere} applied to  $n=N$ elements. Let $0<\epsilon<\epsilon_0$. Then 
\[
e^{\epsilon ih}=\prod_{j=1}^{N} e^{\epsilon z_j'},
\]
where $z_j'\in \ZZ$ for all $j$. Hence, the set $\{e^{\epsilon z}:z\in \ZZ\}^N$  contains all $e^{ih}$, with $h\in \overline{[A,A]}\cap A_{\sa}$ of norm $\leq \epsilon$, and it is thus a neighborhood of the identity in $SU_0(A)$. 
If on the other hand $\epsilon\geq \epsilon_0$, then $\{e^{\epsilon z}:z\in \ZZ\}$ contains 
$\{e^{\frac{\epsilon_0}{2} z}:z\in \ZZ\}$, and so again $\{e^{\epsilon z}:z\in \ZZ\}^N$ is a neighborhood of the identity. 
\end{proof}

\begin{theorem}\label{basisofneighborhoods}
Let $A$ be a unital C*-algebra containing  a full square zero element and having bounded commutators  generation. Let $h\in \overline{[A,A]}\cap A_{\sa}$ be fully noncentral. 
For each $\alpha\in \R$, let
\[
H_\alpha=\{ue^{\pm i\alpha h}u^*:u\in SU_0(A)\}.
\]
Then there exist $n\in \N$ and $\delta>0$ such that the sets $H_\alpha^n$,  
with $0<\alpha \leq \delta$, form a basis of neighborhoods of the identity in $SU_0(A)$.
\end{theorem}

\begin{proof}
By Theorem \ref{finemult}, there exist $n_1\in \N$ and $\delta>0$ 
such that  for each $0<\alpha\leq \delta$ the set $H^{n_1}_\alpha$ contains 
$\{e^{\epsilon z}:z\in \ZZ\}$ for some $\epsilon>0$. Then, with $N\in \N$ as in Lemma \ref{Zneighborhood},  $H_\alpha^{n_1N}$ contains 
$\{e^{\epsilon z}:z\in \ZZ\}^N$, and it is thus a neighborhood of the identity in $SU_0(A)$.

Set $n=n_1N$.  Let $V\subseteq SU_0(A)$ be an arbitrary neighborhood of the identity in $SU_0(A)$. Choose $W$, symmetric neighborhood of the identity invariant under conjugation and such that  $W^n\subseteq V$. Then, $e^{i\alpha h}\in W$ for a small enough $0<\alpha\leq \delta$. It follows that $H_\alpha^n\subseteq V$. Thus, the sets $H_\alpha^n$ form a basis of neighborhoods of the identity in $SU_0(A)$, as desired. 
\end{proof}

Let us recall the invariant Steinhaus property, introduced by Dowerk and Thom (\cite{DowerkThom}).
Let $G$ be a topological group. A subset $W$ of $G$ is called (left) countably syndetic if countably many  left translates of $W$ cover $G$; i.e.,
$G=\bigcup_{k=1}^\infty g_kW$ for some $g_1,g_2,\ldots\in G$. Let $n\in \N$. The group $G$ is said to have the invariant Steinhaus property with exponent
$n\in \N$ if for any $W\subseteq G$ that is countably syndetic, symmetric ($W=W^{-1}$), and invariant under conjugation, $W^n$ is a neighborhood of the identity.

\begin{theorem}
Let $A$ be a unital C*-algebra containing  a full square zero elements and having bounded commutators  generation. Then $SU_0(A)$ has the invariant Steinhaus property.
\end{theorem}

\begin{proof}
Let us choose $h\in \overline{[A,A]}\cap A_{\sa}$ fully noncentral (i.e., $h=[x^*,x]$, with $x$ a full square zero element; see Lemma \ref{fullh}). 
Let $H_\alpha$, $n\in \N$, and $\delta>0$ be as in Theorem \ref{basisofneighborhoods}. Let us show that $SU_0(A)$ has the invariant Steinhaus property with exponent $2n$.

Let $W$ be a subset of $SU_0(A)$  that is symmetric, invariant under conjugation,  and countably syndetic. Say  $SU_0(A) = \bigcup_{k=1}^\infty g_kW$ for some $g_1,g_2,\ldots$ in $ SU_0(A)$. Then for some $k$ the set $g_kW\cap \{e^{i\alpha h}:\alpha\in \R\}$ is uncountable. Using that $W$ is symmetric, it follows that $W^2$ 
contains $e^{i\alpha_0 h}$ for some  $0<\alpha_0\leq \delta$. Hence,  $H_{\alpha_0}^n\subseteq W^{2n}$.
Thus, $W^{2n}$ is a neighborhood of the identity, by Theorem \ref{basisofneighborhoods}.
\end{proof}

\begin{proof}[Proof of Theorem \ref{mainautomatic}]
(i) By the previous theorem, $SU_0(A)$ has the invariant Steinhaus property, which by 
\cite[Proposition 8.10]{DowerkThom} implies that $SU_0(A)$ has the  invariant automatic continuity property. 

(ii) Let $\phi\colon U_0(A)\to U_0(A)$ be a group automorphism. Let us show that $\phi$ is continuous, from which the desired result immediately follows.
 By \cite[Corollary 3]{kallman}, to show that a group isomorphism of polish groups $\psi\colon G_1\to G_2$ is continuous, it suffices to show that for all $U\subseteq G_2$ ranging in a basis of neighborhoods of the identity, $\psi^{-1}(U)$ is an analytic set in $G_2$.

Let $h\in \overline{[A,A]}\cap A_{\sa}$ be fully noncentral (guaranteed to exist by 
Lemma \ref{fullh}). Let 
 $H_\alpha$,  $n\in \N$,  and $\delta>0$ be as in Theorem \ref{basisofneighborhoods}, so that $(H_\alpha^n)_{0<\alpha \leq \delta}$ is a basis neighborhoods of the identity in $SU_0(A)$. Let $v_\alpha=\phi^{-1}(e^{i\alpha h})$. Then 
\[
\phi^{-1}(H_\alpha^n)=\{uv_\alpha^{\pm 1}u^{-1}:u\in SU_0(A)\}^n.
\]
This set is clearly analytic. Hence, $\phi$ is continuous. 
\end{proof}

\begin{corollary}\label{primeclassification}
Let  $\phi\colon U_0(A) \to U_0(B)$ be a group isomorphism, where $A$ and $B$ are prime, unital, traceless C*-algebras
containing full square zero elements. Then $\phi$ is the restriction to $U_0(A)$ of either an  isomorphism or an anti-isomorphism between $A$ and $B$.
\end{corollary}

\begin{proof}
Since $A$ and $B$ are traceless, $SU_0(A)=U_0(A)$ and $SU_0(B)=U_0(B)$. So Theorem \ref{mainautomatic} applies directly to $U_0(A)$ and $U_0(B)$. By the invariant automatic continuity property of these groups,  $\phi$ is  a homeomorphism. Recall that $U_0(A)$ and $U_0(B)$ are the Banach-Lie groups of the Lie algebras of skewadjoint elements $iA_{\sa}$ and $iB_{\sa}$. Thus,
by the functoriality of the Lie algebra of a Banach-Lie group (\cite[Theorem  5.42]{hofmann-morris}), there exists a  Lie algebras homomorphism $\psi\colon iA_{\sa}\to iB_{\sa}$ such that $\phi(e^{ih})=e^{\psi(ih)}$ for all $h\in A_{\sa}$.
Moreover, since $\phi$ is an isomorphism, so is $\psi$. Let us extend  $\psi$ to $A$ by setting 
\[
\psi(a+ib)=-i\psi(ia)+\psi(ib),
\] 
for $a,b\in A_{\sa}$. We readily check that  $\psi$ is again a Lie algebras isomorphism between $A$ and $B$. By \cite[Theorem 6.5.24]{ara-mathieu} (alternatively, by \cite{bresar}),
$\phi$ is either an isomorphism or an anti-isomorphism of C*-algebras. In either case we get that
\[
\phi(e^{ih})=e^{\psi(ih)}=\psi(e^{ih})
\]
for all $h\in A_{\sa}$. It follows that $\phi$ is the restriction of $\psi$ to $U_0(A)$.
\end{proof}

If $A$ is a  unital traceless C*-algebra containing a full square zero element, then the equality $SU_0(A)=U_0(A)$ and Theorem \ref{mainautomatic} imply that  $U_0(A)$ does not have discontinuous automorphisms. This is in contrast with the tracial case.

\begin{theorem}\label{discontinuousauto}
Let $A$ be a separable unital C*-algebra with at least one tracial state. Then $U_0(A)$ admits discontinuous automorphisms.
\end{theorem}

\begin{proof}
We adapt the proof of the same result for $U_0(M_n(\C))$ from \cite{kallman}. 

Let $\tau\colon A\to \C$ be a tracial state. Let $\Gamma\subseteq \T$ denote the group of scalar unitaries in the kernel of $\Delta_\tau$. From the description of the kernel of $\Delta_\tau$ \eqref{Dtauker} we see that 
\[
\Gamma=\{e^{2\pi i\theta}: \theta\in \tau(K_0(A))\}.
\]
Since $A$ is separable, this group is countable (as $K_0(A)$ is countable).  Let us choose a discontinuous automorphism $\alpha \colon \T\to \T$ that fixes $\Gamma$. (To get one, let $V$ be the vector subspace of $\R$  spanned by $\tau(K_0(A))$, where the scalar field is $\Q$. Choose a basis of $V$, a fortiori a countable set, and extend it to a Hamel basis of $\R$. Using this basis, choose a $\Q$-linear transformation $f\colon \R\to \R$ that is the identity on $\tau(K_0(A))$, but that is otherwise discontinuous.
Now define $\alpha(e^{2\pi i\theta})=e^{2\pi i f(\theta)}$, which  is well defined since $f$ is the identity on $\Z\subseteq \tau(K_0(A))$.)  

Any unitary  $v\in U_0(A)$ is expressible in the form $zu$, where $z\in \T$ and $u\in \ker \Delta_\tau$.  
To see this, choose any path $\eta\colon [t_1,t_2]\to U_0(A)$ connecting $1$ to $v$, and set $z=e^{2\pi i\theta}$, with $\theta=\tilde\Delta_\tau(\eta)$.
Then $v=zu$ with  $u\in \ker\Delta_\tau$.

Let us define $\tilde\alpha\colon U_0(A)\to U_0(A)$ by $\tilde\alpha(zu)=\alpha(z)u$, where $z\in \T$ and $u\in \ker \Delta_\tau$.  We readily verify that
this is a well defined map: If $z_1u_1=z_2u_2$, with $z_1,z_2\in \T$ and $u_1,u_2\in \ker \Delta_\tau$, then 
\[
z_1z_2^{-1}=u_2u_{1}^{-1}\in \Gamma.
\] 
Since $\alpha$ is the identity on $\Gamma$, $\alpha(z_1z_2^{-1})=u_2u_{1}^{-1}$, which in turn implies that
$\alpha(z_1)u_1=\alpha(z_2)u_2$. Hence $\tilde\alpha$ is well defined. It is then easily shown that $\tilde\alpha$ is a bijective group homomorphism. Since the restriction of $\tilde\alpha$ to the scalar unitaries agrees with $\alpha$,  $\tilde\alpha$ is necessarily discontinuous. 
\end{proof}

\section{Bounded normal generation}
\begin{theorem}\label{mainlocalBNG}
Let $A$ be a unital C*-algebra containing  a full square zero element and having bounded commutators  generation.  Let $H\subseteq SU_0(A)$ be a symmetric,  fully noncentral set invariant under conjugation. Then there exists $n\in \N$ such that $H^n$ is a neighborhood of the identity in $SU_0(A)$.
\end{theorem}

\begin{proof}
By Theorem \ref{mainmult}, there exists $n_1\in \N$ such that $H^{n_1}$ contains $\{e^{z}:z\in \ZZ\}$, where $\ZZ\subseteq iA_{\sa}$ as defined in \eqref{Zset}.  
On the other hand, by  Lemma \ref{Zneighborhood} there exists $N\in\N$ such that  $\{e^{z}:z\in \ZZ\}^N$ is a neighborhood of the identity in $SU_0(A)$. It follows that $H^{n_1N}$ is also  a neighborhood of the identity, thus proving the theorem.  
\end{proof}

Specializing the theorem above to simple C*-algebras, we obtain Theorem \ref{mainBNG}
from the introduction.

\begin{proof}[Proof of Theorem \ref{mainBNG}]
If $A=\C$, the theorem is trivial. Assume thus that this is not the case. Observe then that $A$ contains nonzero square zero elements, by Glimm's halving lemma, and these are automatically full since $A$ is simple.

(i) Let $g$ be an element in  $SU_0(A)/Z(SU_0(A))$ distinct from the identity. Let $w\in SU_0(A)$ be any lift of $g$, necessarily noncentral (and fully noncentral, as $A$ is simple). Let $H=\{uw^{\pm1}u^*:u\in SU_0(A)\}$. By Theorem \ref{mainlocalBNG}, there exists $n\in \N$ such that $H^n$
contains a neighborhood of the identity in $SU_0(A)$. Then $H^n$ is mapped onto a neighborhood of the identity by the quotient map. This establishes local bounded normal generation for $SU_0(A)/Z(SU_0(A))$.

(ii) Suppose that  $SU_0(A)$ is bounded as a metric space. Claim: For any neighborhood of the identity $V\subseteq SU_0(A)$ there exists $m\in \N$ such that $SU_0(A)=V^m$. Proof: 
By \cite[Theorem A]{ando}, the boundedness of $SU_0(A)$ under the exponential length metric implies that $SU_0(A)$
is coarsely bounded. This, by \cite[Theorem 1.4]{rosendalOB}, implies that  there exist a finite set $F\subseteq SU_0(A)$ and $k\in \N$ such that $SU_0(A)=(FV)^k$. Since $SU_0(A)$
is also connected, the finite set $F$ is contained in $V^{k'}$ for a large enough $k'$. Hence $SU_0(A)=V^{(k'+1)k}$,  proving the claim.  

Let $g$ be an element in  $SU_0(A)/Z(SU_0(A))$ distinct from the identity. Let us apply the claim just established  to $H^n$, with $H$ and $n$ as in the proof of (i).  Then $U_0(A)=H^{nm}$. Passing to the quotient, we obtain bounded normal generation for $SU_0(A)/Z(SU_0(A))$.
\end{proof}

A class of C*-algebras to which all results of this section and the previous section apply is the purely infinite simple C*-algebras. 
\begin{corollary}\label{coropi}
Let $A$ be a simple, unital, purely infinite C*-algebra. Then $U_0(A)/\T$ has bounded normal generation, $U_0(A)$ has the automatic invariant continuity property, and if $A$ is separable then 
$U_0(A)$ has a unique polish group topology.
\end{corollary}

\begin{proof}
Since $A$ is traceless, $A=[A,A]$ and $A$ has bounded commutators generation,  by Pop's theorem. Hence,  $SU_0(A)=U_0(A)$. Moreover, by \cite[Proposition 9]{phillips}, the exponential length of $A$ is at most $\pi$. Hence, $U_0(A)$ is bounded. The previous theorem then implies that $U_0(A)/\T$ has bounded normal generation. The automatic invariant continuity property and, in the separable case,  the uniqueness of the polish group topology, follow from Theorem \ref{mainautomatic}.   
\end{proof}

In Theorem \ref{BCGclasses} we have already recalled various classes of C*-algebras that have bounded commutators generation. We now turn to the question of boundedeness of $SU_0(A)$, in order to produce more examples of C*-algebras where the group $SU_0(A)/Z(SU_0(A))$ has bounded normal generation.

\begin{lemma}\label{N2normreduction}
Let $A$ be a unital C*-algebra. Then $\el_{SU_0(A)}(e^{i[x^*,x]})\leq \pi$ for all $x\in \Ntwo$.
\end{lemma}

\begin{proof}
If $x=0$ the lemma is trivially true, so assume that this is not the case. Write $x=Cy$,  with $y\in \mathcal N_2$ of norm 1 and $C=\|x\|$. It suffices to prove the lemma in the case that $y$ is the canonical generator in the universal C*-algebra generated by a square zero element of norm $\leq 1$. This C*-algebra is $M_2(C_0(0,1])$, with $y=\begin{pmatrix} 0&t\\0&0\end{pmatrix}$. Let us thus assume that we are in this set-up. Then 
\[
e^{i[x^*,x]}=e^{iC^2[y^*,y]}=\begin{pmatrix} e^{iC^2t^2}&0\\0&e^{-iC^2t^2}\end{pmatrix}.
\]
This is a unitary in the unitization of  $M_2(C_0(0,1])$.

Let $0<t_1<t_2<\cdots<t_n\leq 1$ be the points where $Ct^2$ is an integer multiple of $\pi$. 
Choose $\delta>0$ such that the intervals $(t_k-2\delta,t_k+2\delta)$, with $k=1,\ldots,n$, are pairwise disjoints. 
Define $f\in C_0((0,1])_+$ as follows: 
\[
f(t)=\begin{cases}
Ct_k^2&\hbox{if $t\in (t_k-\delta, t_k+\delta)$ for some $k$,}\\
Ct^2 &\hbox{if $t\notin \bigcup_k (t_k-2\delta,t_k+2\delta)$,}\\
\hbox{linear}&\hbox{for $t\in (t_k-2\delta, t_k-\delta)$ and $t\in (t_k+\delta,t_k+2\delta)$ and all $k$.} 
\end{cases}
\]

Let $\epsilon>0$. Choose $\delta$ smaller if necessary,  so that the function  $f$ defined above satisfies that $|f(t)-Ct^2|<\epsilon$ for all $t\in [0,1]$.
Now let $g(t)=f(t)-Ct^2$ for $t\in [0,1]$, so that
\[
Ct^2 = f(t) + g_+(t) - g_-(t). 
\]
Set 
\[
x'=\begin{pmatrix}
0 & f^{\frac 12}\\
0 & 0
\end{pmatrix},\quad
y_1=\begin{pmatrix}
0 & g_+^{\frac12}\\
0 & 0
\end{pmatrix},\quad
y_2 = y_1=\begin{pmatrix}
0 & g_-^{\frac12}\\
0 & 0
\end{pmatrix}.
\]
It is straightforward to verify that 
\[
[x^*,x]=[(x')^*,x']+[y_1^*,y_1]+[y_2^*,y_2].
\]
Since the summands on the right-hand side commute, we have that
\[
e^{i[x^*,x]}=e^{i[(x')^*,x']}e^{i[y_1^*,y_1]}e^{i[y_2^*,y_2]}.
\]
Observe that $\|y_1\|=\|g_+^{\frac12}\|<\epsilon^{\frac 12}$, so that $\el_{SU_0(A)}(e^{i[y_1^*,y_1]})<\epsilon$. Similarly, we have that $\el_{SU_0(A)}(e^{i[y_2^*,y_2]})<\epsilon$. 

To complete the proof,   let us show that   $e^{i[(x')^*,x']}$ is unitarily equivalent to 
$e^{i[(x'')^*,x'']}$ for  some $x''$ such that $\|x''\|\leq \sqrt{\pi}$. Define $f_1\in C(0,1]$
such that $0\leq f_1(t)\leq \pi$ for all $t\in [0,1]$,  and either $f(t)-f_1(t)$ or $f(t)+f_1(t)$ is an integer multiple
of $\pi$ for all $t$. Let 
\[
x''=\begin{pmatrix}
0 & f_1^{\frac 12}\\
0 & 0
\end{pmatrix}.
\]
Then $\|x''\|\leq \sqrt{\pi}$ and
\[
e^{i[(x')^*,x']}=\begin{pmatrix} e^{if}&0\\0&e^{-if}\end{pmatrix},\quad 
e^{i[(x'')^*,x'']}= \begin{pmatrix} e^{if_1}&0\\0&e^{-if_1}\end{pmatrix}.
\] 
These two unitaries are  unitarily equivalent. This stems from the fact that the unordered pairs $\{e^{if(t)},e^{-if(t)}\}$ and $\{e^{if_1(t)},e^{-if_1(t)}\}$ agree for all $t$, and further,
the functions $e^{if(t)}$ agree $e^{-if(t)}$ on an open neighborhood of the values of $t$ at which they cross. A unitary $t\mapsto w(t)$ conjugating $e^{i[(x')^*,x']}$ and $e^{i[(x'')^*,x'']}$ is gotten by setting $w(t)$ equal to either the identity matrix or the matrix
$\begin{pmatrix}0 & 1\\1& 0\end{pmatrix}$ for $t$ outside the intervals $(t_k-\delta,t_k+\delta)$, and inside these intervals $t\mapsto w(t)$ is defined as any path connecting these two matrices. 
\end{proof}

\begin{lemma}\label{nuc1}
Let $A$ be a unital C*-algebra of nuclear dimension at most 1. Then $\el_{SU_0(A)}(e^{ih})\leq 4\pi$ for all $h\in \overline{[A,A]}\cap A_{\sa}$.
\end{lemma} 
  
 \begin{proof}
 Let $h\in \overline{[A,A]}\cap A_{\sa}$.
 By \cite[Theorem 3.2]{ker-det}, 
 there exist 
sequences $h_n^{(1)},h_n^{(2)},h_n^{(3)},h_n^{(4)}$ in $\overline{[A,A]}\cap A_{\sa}$ such that
\begin{enumerate}
\item
$h_n^{(1)}+h_n^{(2)}+h_n^{(3)}+h_n^{(4)}\to h$,
\item
$[h_n^{(1)}+h_n^{(2)},h_n^{(3)}+h_n^{(4)}]\to 0$, $[h_n^{(1)},h_n^{(2)}]\to 0$,
$[h_n^{(3)},h_n^{(4)}]\to 0$,
\item
$h_n^{(k)}=[(x_n^{(k)})^*,x_n^{(k)}]$ with $x_n^{k}\in \Ntwo$  for  $k=1,2,3,4$ and all $n$.
\end{enumerate}
We now use \cite[Lemma 2.2]{ker-det}, which  states that if $a_n,b_n$ are bounded sequences of selfadjoint elements
such that $[a_n,b_n]\to 0$ then $e^{i(a_n+b_n)}e^{-ia_n}e^{-ib_n}\to 0$
in $SU_0(A)$. Using this lemma (repeatedly) we deduce that
\[
e^{ih_n^{(1)}}e^{ih_n^{(2)}}e^{ih_n^{(3)}}e^{ih_n^{(4)}}\to e^{ih}
\]
in $SU_0(A)$. On the other hand, by Lemma \ref{N2normreduction}, the exponential length in $SU_0(A)$ of $e^{ih_n^{(k)}}$ is bounded by $\pi$. The lemma thus follows.
 \end{proof}

Let $A$ be a unital C*-algebra. The exponential rank of $A$ is defined as the least $m\in \N$ such that any unitary $u\in U_0(A)$ is a product of at most $m$ exponential unitaries (i.e., of the form $e^{ih}$ with 
$h\in A_{\sa}$). If no such $m$ exists then the exponential rank of $A$ is set equal to $\infty$.  If $A$ has exponential rank $m+1$, and the products of $m$ exponentials form a dense subset of $U_0(A)$, then $A$ is also said to have exponential rank $m+\epsilon$. 

Next, let us recall the property of strict comparison by traces in a simple unital C*-algebra.
(There is also a version of this property for arbitrary C*-algebras; e.g., see \cite{NgRobertpure}.) Let  $A$ be a simple unital C*-algebra. Then $A$ said to have strict comparison of positive elements by traces if
for all   $a,b\in (A\otimes \mathcal K)_+$ such that $d_\tau(a)<d_\tau(b)$ for all tracial states $\tau$
on $A$ we have that $a\precsim_{\Cu} b$. (Here $d_\tau(c):=\lim_n \tau(c^{\frac1n})$ and $\precsim_{\Cu}$ denotes the Cuntz comparison relation, i.e., $d_n^*bd_n\to a$ for some sequence $d_n\in A\otimes \mathcal K$.) This property, introduced by Blackadar in \cite{blackadar},  is a C*-algebraic analogue of the ``comparison of projections by traces" properties of factors.

\begin{theorem}
Let $A$ be a simple unital C*-algebra. Suppose that $A$ has stable rank one, strict comparison by traces, and finite exponential rank. Then $SU_0(A)$ is bounded. Consequently, $SU_0(A)/Z(SU_0(A))$ has BNG. 
\end{theorem}

\begin{proof}
If $A$ is a matrix algebra $M_n(\C)$, the theorem is well known. Let us thus  assume that $A$ is non-elementary.
Then, under the hypotheses of the theorem, $A$ is a pure C*-algebra, i.e., its Cuntz semigroup is almost unperforated and almost divisible. Almost unperforation of the Cuntz semigroup  is implied by  strict comparison of positive elements by traces (\cite{ERS}), which $A$ is assumed to have. For simple non-elementary unital C*-algebras of stable rank one, almost divisibility in the Cuntz semigroup follows automatically from almost unperforation, by Thiel's \cite[Theorem 8.11]{Thiel}. Since $A$ is pure and has strict comparison of positive elements by traces, it has bounded commutators generation
(Theorem \ref{BCGclasses}). Thus, the BNG property for $SU_0(A)/Z(SU_0(A))$ will indeed follow once we have shown that 
$SU_0(A)$ is bounded.

To show that $SU_0(A)$ is bounded  we shall apply the ``special subalgebra technique" also employed in \cite{ker-det},   \cite{ART}, \cite{jacelon}. More specifically, by \cite[Theorem 4.1]{ker-det}, there exists a C*-subalgebra $B\subseteq A$ such that
\begin{enumerate}
\item
$[B,B]=[A,A]\cap B$, 
\item
$B$ has nuclear dimension 1 (in fact, $B$ is of the form $C\otimes W$, where $C$ is an AF C*-algebra and $W$ is  the Jacelon-Razak algebra), 
\item
every non-invertible selfadjoint element $h\in A$ with connected spectrum is approximately unitarily equivalent to some $h'\in B$.  
\end{enumerate}

The proof now proceeds along a series of claims, each time expanding the set of elements in $SU_0(A)$ on which the exponential length $\el_{SU_0(A)}$ is bounded.

\emph{Claim 1}: $\el_{SU_0(A)}$ is  bounded on the set of  elements of the form $e^{ih}$, with $h\in [B,B]\cap B_{\sa}$. Proof:
The C*-algebra $B^\sim$  has nuclear dimension 1.  The claim now follows from Lemma \ref{nuc1}.

\emph{Claim 2}:  $\el_{SU_0(A)}$ is bounded on the set of elements of the form $e^{ih}$, with  $h\in [A,A]\cap A_{\sa}$ non-invertible and with connected spectrum. 
Proof: Let $h\in [A,A]\cap A_{\sa}$ be non-invertible and with connected spectrum. Then,  by the properties of $B$ listed above, $h$ is approximately unitarily equivalent to some $h'\in B$. Say $u_nh'u_n^*\to h$.
Then 
\[
\el_{SU_0(A)}(e^{iu_nh'u_n^*})\to \el_{SU_0(A)}(e^{ih}),
\] 
while on the other hand
\[
\el_{SU_0(A)}(e^{iu_nh'u_n^*})=\el_{SU_0(A)}(u_ne^{ih'}u_n^*)=\el_{SU_0(A)}(e^{ih'}), 
\]
since  $\el_{SU_0(A)}$ is invariant under automorphisms of $A$.  Hence, $\el_{SU_0(A)}(e^{ih})=\el_{SU_0(A)}(e^{ih'})$.  The claim now follows from the previous claim.

\emph{Claim 3}: $\el_{SU_0(A)}$ is bounded on the set of  elements  of the form $e^{ih}$, with $h\in [A,A]\cap A_{\sa}$. Proof: Consider a unitary $e^{ih}$, with $h\in [A,A]\cap A_{\sa}$. By  \cite[Lemma]{ker-det}, there exist $h'\in  [A,A]\cap A_{\sa}$ and $x\in \Ntwo$ such that 
$h'$ is a selfadjoint element with connected spectrum, $h'$ and $[x^*,x]$ belong to $C^*(h)$, and
\[
\|h -  (h'+[x^*,x])\|<1.
\] 
Since $h$, $h'$, and $[x^*,x]$ commute with each other, $e^{ih}=e^{ih'}e^{[x^*,x]}e^{ic}$, where $c\in [A,A]\cap A_{\sa}$ and $\|c\|\leq 1$. Observe that  $\el_{SU_0(A)}$
is uniformly bounded on the second and third factors (by Lemma \ref{N2normreduction}). It remains to show that $\el_{SU_0(A)}$ is bounded on the  set of exponentials $e^{ih}$
where $h\in [A,A]\cap A_{\sa}$ has connected spectrum. If $\|h\|\leq 2\pi$, we are done. If $\|h\|>2\pi$, then there is a  translate  of $h$ by an integer multiple of $2\pi\cdot 1$ that is non-invertible, and again we are done, by the previous claim.  This proves the claim.

\emph{Claim 4}:  $\el_{SU_0(A)}$ is bounded on commutators of the form $(e^{ih},w)$, with $h\in A_{\sa}$ and $w\in U_0(A)$: Proof:
Consider an element of the form $(e^{ih},w)=e^{ih}e^{-iwhw^*}$. By \cite[Lemma~6.4]{KirchbergRordam} there exists an  approximately central sequence of selfadjoint elements $(h_n)_n$ such that $h-h_n\in [A,A]$ and  $\|h_n\|\leq \|h\|$ for all $n$. Then, using \cite[Lemma 2.2]{ker-det}, for large enough $n$ we have that
\begin{align*}
e^{ih} =e^{ic_1}e^{i(h-h_n)}e^{ih_n},\quad e^{ih_n}e^{-iwh_nw^*}=e^{ic_2}
\end{align*}
for some $c_1,c_2\in [A,A]\cap A_{\sa}$ of norm $\leq 1$. Hence,
\[
e^{ih}e^{-iwhw^*}=e^{ic_1}e^{i(h-h_n)}e^{ic_2}e^{-iw(h-h_n)w^*}e^{-iwc_1w^*}.
\]
By the previous claim, $\el_{SU_0(A)}$ is uniformly bounded on the terms $e^{i(h-h_n)}$ and $e^{-iw(h-h_n)w^*}$. The claim thus follows.

The following claim completes the proof of the theorem. 

\emph{Claim 5}: $\el_{SU_0(A)}$ is bounded on all $SU_0(A)$. 
Let $R$ denote the exponential rank of $A$.	The C*-algebra $A$ fulfills the assumptions of  \cite[Theorem 1.1]{ker-det}.  Thus, by  this theorem, any element of $SU_0(A)$ can be expressed as a product of $7R+29$ commutators $(v_k,w_k)$, with
$v_k,w_k \in U_0(A)$. It thus suffices to show that $\el_{SU_0(A)}$ is bounded on the set of commutators $(v,w)$,  with $v,w\in U_0(A)$. Consider a pair $v,w\in U_0(A)$. Since $A$ has  exponential rank $R$,  we can write  $v=\prod_{k=1}^R e^{ih_k}$, where $h_k\in A_{\sa}$ for all $k$. Repeatedly using the identity
$(xy,w)=(x,y)(y,xw)(x,w)$, we can express the commutator $(\prod_{k=1}^R e^{ih_k},w)$ as a product of commutators of the form $(e^{ih_k}, w')$.	
It thus suffices to prove that $\el_{SU_0(A)}$ is bounded on all commutators of the form $(e^{ih},w)$, with $h\in A_{\sa}$ and $w\in U_0(A)$. This is what we have proven in the previous claim. 
\end{proof}

Let $\mathcal Z$ denote the Jiang-Su C*-algebra. A C*-algebra $A$ is called $\mathcal Z$-stable if $A\otimes \mathcal Z\cong A$.

\begin{corollary}\label{classifiable}
Let $A$ be a unital separable simple nuclear $\mathcal Z$-stable C*-algebra satisfying the UCT. Then $SU_0(A)/Z(SU_0(A))$ has BNG.  	
\end{corollary}	
\begin{proof}
By R{\o}rdam's dichotomy \cite{RordamZ}, a simple unital $\mathcal Z$-stable C*-algebra  is either purely infinite or of stable rank one. The purely infinite case  was dealt with in Corollary \ref{coropi}. Let us thus assume that $A$ has stable rank one.  Then $\mathcal Z$-stability implies strict comparison by 2-quasitraces, and nuclearity implies that 2-quasitraces are traces by Haagerup's theorem (\cite{haagerup}). So $A$ has strict comparison of positive elements by traces.  By the previous theorem, it remains  to check that $A$ has finite exponential rank. We prove this next. (Note: There is abundant evidence that  classifiable C*-algebras have exponential rank $1+\epsilon$, e.g. \cite{LinTR1}, but we will content ourselves here with a finite bound.)

By the classification theorem of Gong, Lin, and Niu in \cite{GLN1}, together with  \cite{EGLN2}, \cite{TWW}, and the work on the Toms-Winter in \cite{CETWW}, the C*-algebra  $A$ is an inductive limit of subhomogeneous C*-algebras $(A_n)_{n=1}^\infty$ with spectra of dimension at most 2. Since the exponential rank is easily seen not to increase by more than ``$\epsilon$'' when passing to an  inductive limit, we will be done once we have  shown that the C*-algebras $(A_n)_{n=1}^\infty$ have uniformly bounded exponential rank. 

The C*-algebra $A_n$ is a pullback  of the following form (see \cite[Section 13]{GLN1} and \cite[Section 2]{XinLi}): 
\[
\xymatrix{
A_n\ar[r]^{\psi}\ar[d]^{\phi}  & C([0,1],E)\ar[d]^{\mathrm{ev_0}\oplus\mathrm{ev_1}}\\
B \ar[r]^{\beta} & E\oplus E
}
\] 
In this diagram  
\begin{itemize}
\item
$B=pM_N(C(X))p$ where $p$ is a projection and $X$ is a compact metric space of  dimension at most $2$, 
\item
$E$ is finite dimensional, 
\item
$\mathrm{ev}_0$ and $\mathrm{ev}_1$ are the point evaluations at the endpoints, 
\item
$\beta$ is a unital homomorphism. 
\end{itemize}

Let us show that these C*-algebras have uniformly bounded exponential rank. Let $u\in U_0(A_n)$, and let $[0,1]\ni t\mapsto u_t$ be a path connecting $1$ to $u$. 
Consider the path $t\mapsto \phi(u_t)$ in $U_0(B)$. By \cite[Theorem 4.7]{phillipsHowMany}, there exists a constant $C(3)$ bounding the exponential ranks of the C*-algebras that are a corner of a matrix algebra over a compact matrix space of dimension $\leq 3$. Applied to the C*-algebra  
$C([0,1],B)$,  we get that there exist selfadjoint elements  $h_k\in C([0,1],B)$, for $k=1,\ldots,C(3)$, such that 
\[
\phi(u_t)=\prod_{k=1}^{C(3)} e^{ih_k(t)}
\] 
for all $t\in [0,1]$. The map $\phi$ is surjective (since $\mathrm{ev}_0\oplus\mathrm{ev}_1$ is), and so the induced map $C([0,1],A_n)\to C([0,1],B)$ is also surjective. Let 
$\tilde h_k\in C([0,1],A_n)$ be a selfadjoint lift of $h_k$ for all $k$. Set 
\[
v=u\prod_{k=1}^{C(3)}e^{-i\tilde{h}_k(1)}.
\]
Observe then $v$ is connected to $1$ in $U_0(A_n)$ by a path that is mapped to $1\in B$ by $\phi$. To complete the proof we show that $v$ is a product of two exponentials. Observe that $\psi(v)\in C([0,1],E)$ is a unitary satisfying the endpoint conditions $\psi(v)(0)=\psi(v)(1)=1_E$. It thus belongs to the subalgebra $I=\{f\in C([0,1,E]):f(0)=f(1)\in \C 1_E\}$. Moreover, $\psi(v)$ is connected to $1$ by a path of unitaries in $I$. Since $I$ has exponential rank $1+\epsilon$ by \cite[Proposition 2.8]{phillipsFU}, there exist selfadjoint elements $g_1,g_2\in I$ such that  $\psi(v)=e^{ig_1}e^{ig_2}$. Let $\tilde g_1,\tilde g_2\in A_n$ be their lifts that are  mapped to $1\in B$ by $\phi$. Then $v=e^{i\tilde h_1}e^{i\tilde g_2}$, as desired. 
\end{proof}

\begin{remark}
In \cite{dowerkthesis} Dowerk introduced the property of ``topological bounded normal generation''. A topological group $G$ has this property if for any $g\in G$ distinct from the unit there exists $n\in \N$ such that $\{hg^{\pm 1}h^{-1}:h\in G\}^n$ is dense in $G$.
In \cite{DowerkLeMaitre} examples are given of  nonsimple groups  with topological bounded normal generation. The C*-algebras covered in Corollary \ref{classifiable} provide us with new examples of such groups. Take $A$ an infinite dimensional simple unital AF C*-algebra. Since $A$ has at least one tracial state, there are noncentral unitaries in $U_0(A)$ with nonzero determinant. Hence,  $U_0(A)/\T$ is not simple, as  $DU_0(A)/\T$ is a proper normal subgroup of $U_0(A)/\T$. In particular, $U_0(A)/\T$ does not have BNG. Let us show that $U_0(A)/\T$ has topological bounded normal generation. Let $u\in U_0(A)$ be noncentral. Then $u'=(u,e^{ih})$ is noncentral for some $h\in A_{\sa}$. By the bounded normal generation of $DU_0(A)$, there exists $n$ such that 
\[
DU_0(A)=(\{v(u')^{\pm 1}v^*:v\in DU_0(A)\})^n\subseteq 
(\{vu^{\pm 1}v^*:v\in DU_0(A)\})^{2n}.
\]
Since $A$ has real rank zero,  $U_0(A)=\overline{DU_0(A)}^{\|\cdot\|}$ (\cite{elliott-rordam0}). Thus, $(\{vu^{\pm 1}v^*:v\in DU_0(A)\})^{2n}$  is dense in $U_0(A)$. This shows that 
$U_0(A)/\T$ has topological bounded normal generation.
\end{remark}

\section{Counterexamples}\label{seccounterexamples}
Let us first introduce notation and recall a known lemma on the exponential length in $U_0(A)$. Let $A$ be a unital C*-algebra. Given $u\in U_0(A)$, we define its exponential length as
\[
\el_{U_0(A)}(u)=\inf \Big\{\sum_{j=1}^n \|h_j\|: u=\prod_{j=1}^n e^{ih_j},\, h_j\in A_{\sa}\hbox{ for all }j\Big\}.
\]
Clearly, for $u\in SU_0(A)$ we have that $\el_{U_0(A)}(u)\leq \el_{SU_0(A)}(u)$ (the latter defined in \eqref{elSU0}). These two quantities are often different. We call $(u,v)\mapsto \el_{U_0(A)}(u^*v)$ the exponential distance in $U_0(A)$.

\begin{lemma}\label{ringroselemma}
Let $m\in \N$ and $0< \theta<\pi$. Let $A$ be a unital C*-algebra and let $u\in U_0(A)$ be such that  $\el_{U_0(A)}(u)<m\theta$. Then there exist $h_1,\ldots,h_m\in A_{\sa}$ such that  $u=\prod_{j=1}^m e^{ih_j}$ 
and   $\|h_j\|<\theta$ for all $j$.
\end{lemma}
\begin{proof}
This is essentially \cite[Theorem 2.6]{Ringrose}. By assumption,  
\[
u=\prod_{j=1}^n e^{if_j},
\]
were $f_j\in A_{\sa}$ for all $j$ and  $\sum_{j=1}^n \|f_j\|<m\theta$. Increasing $n$ if necessary, we can assume that  $\|f_j\|<\delta$ for all $j$, where we choose $\delta>0$ such that $0<\frac{m\theta}{\theta-\delta}<m+1$. Let us now break-up the list  $e^{if_1},\ldots,e^{if_n}$ into groups of consecutive exponentials such that the norms of the exponents in each group add up to less than $\theta$, but  $\geq \theta-\delta$. It is then easily derived  that the number of groups obtained in this way must be at most $m$.  The product of the exponentials in each of these groups can be written as $e^{ih}$, with $\|h\|<\theta$,  by \cite[Corollary 2.2]{Ringrose}. We thus have $u=\prod_{j=1}^{m'} e^{ih_j}$, with $\|h_j\|<\theta$ for all $j$ and $m'\leq m$. Setting $h_j=0$ for $m'<j\leq m$ the lemma readily follows.
\end{proof}

\subsection{Homogeneous C*-algebras}
Let $X$ be a compact Hausdorff space (below we will fix $X=(S^2)^n$). Let us introduce notation and recall some standard  facts around the projections in $M_\infty(C(X))$.

Let   $p,q\in M_\infty(C(X))$ be projections. Then $p$ and $q$ are said to be  Murray-von Neumann equivalent if $p=v^*v$ and $q=vv^*$ for some $v\in M_\infty(C(X))$. We denote the Murray-von Neumann class of a projection $p\in M_\infty(C(X))$ by $[p]$. 

Suppose that $p\in M_m(C(X))$ and  $q\in M_n(C(X))$. Let us  denote by $p\oplus q$ the projection 
\[
\begin{pmatrix}p&0\\0&q\end{pmatrix}\in M_{n+m}(C(X)).
\] 
One has that  $[p\oplus q]$ depends only on $[p]$ and $[q]$. Setting
$[p]+[q]:=[p\oplus q]$ defines the addition operation in the Murray-von Neuman monoid of projections.

Let us denote by $\her(p)$ the C*-algebra $pM_\infty(C(X))p$, i.e., the hereditary C*-subalgebra generated by $p$. We adopt a 2x2 matrix notation for the elements of 
$\her(p\oplus q)$. More specifically, we represent  $a\in \her(p\oplus q)$ as the matrix
\[
\begin{pmatrix}
pap & paq\\
qap & qaq
\end{pmatrix}.
\]

We will make reference to the complex vector bundle associated to a projection. The complex vector bundle on $X$ associated to $p$ is $(X,E,\pi)$, where 
\[
E=\{(x,v)\in X\times \C^m:p(x)v=v\},
\] 
and $\pi\colon E\to X$ is the projection onto the first coordinate. We denote this vector bundle by $\eta_p$. The isomorphism class of $\eta_p$
depends only on $[p]$.

Let  $p\in M_\infty(C(X))$ be a projection. Let us denote by $e(p)\in H^*(X)$ the Euler class of the vector bundle $\eta_p$. If $p$ has rank $k$, then $\eta_p$   has dimension $k$ over $\C$, and $2k$ over $\R$, so  $e(p)\in H^{2k}(X)$.

Let $n\in \N$. From this point on we fix $X=(S^2)^n$, where $S^2$ denotes the 2-dimensional sphere.

Let us recall a few standard facts on the homology and cohomology groups (with integer coefficients) of $X$. By Kunneth's theorem,  
\[
H^*(X)\cong Z[\alpha_1,\alpha_2,\ldots,\alpha_n]\,/\,(\alpha_j^2:j=1,\ldots,n),
\] 
where $\alpha_1,\ldots,\alpha_n\in H^2(X)$ are induced by the projections $\pi_j\colon X\to S^2$ onto each sphere factor
and the choice of a generator in $H^2(S^2)$.  

For each choice of indices $i_1,i_2,\ldots,i_k$, let $S_{i_1,i_2,\ldots,i_k}\subseteq X$ denote the cartesian product with $S^2$ at the indices $i_1,i_2,\ldots,i_k$ and a fixed point (choose any) at the remaining indices. The inclusion $S_{i_1,i_2,\ldots,i_k}\subseteq X$ gives rise to homology classes $[S_{i_1,i_2,\ldots,i_k}]\in H_*(X)$, image of the fundamental homology class of $S_{i_1,i_2,\ldots,i_k}$ in $H^*(X)$. By Poincare duality, we have that $H_{n-k}(X)\cong H^k(X)$ for $k=0,\ldots,n$, where the isomorphism is given by $\alpha\mapsto \alpha \cap [X]$ (cap product with the fundamental homology class of $X$).  Moreover, we can compute that 
\[
(\alpha_{i_1}\alpha_{i_2}\cdots \alpha_{i_k})\cap [X] = [S_{i_1',i_2',\ldots ,i_{n-k}'}],
\]
where the indices on the right side run through the complement of the indices on the left side.

Let  $\kappa\colon X\to [-1,1]$ be defined as  
\begin{equation}\label{kappafunction}
\kappa((x_1,y_1,z_1), \dots, (x_n,y_n,z_n))=x_1.
\end{equation}

\begin{theorem}\label{construction}
Let $k,m\in \N$. 	Let $p,q\in M_\infty(C(X))$ be projections such that
$[p]\leq k[1]$, the Euler class $e(q)$  belongs to the subring of $H^*(X)$ generated by $\alpha_2,\ldots,\alpha_n$, and further $e(q)^{5km}\neq 0$. Consider a selfadjoint element 
	$h\in \her(p\oplus q)$ of the form
	\[
	h=\begin{pmatrix}
	\kappa p& 0\\0 & w
	\end{pmatrix}.
	\]
	Then the exponential distance in $U_0(\her(p\oplus q))$ from $e^{ith}$ to the set 
	\[
	(\{(v,w):v,w\in U_0(\her(p\oplus q))\})^{m}
	\] is at least $\min (\frac{t}k, (m+1)\pi)$ for all $t\geq 0$.
\end{theorem}	

\begin{proof}
Since we may replace $q$ by a Murray-von Neumann equivalent projection, let us  assume that $q$ is a smooth function on $X$.

Since $[p]\leq k[1]$,  $p$ is Murray-von Neumann subequivalent to $1_k$ (the unit in $M_k(C(X))$). Say $p=v^*v$ and $p':=vv^*\leq 1_k$ for some $v\in M_\infty(C(X))$.
Then $a\mapsto (v\oplus q)a(v^*\oplus q)$ is a C*-algebra isomorphism from $\her(p\oplus q)$ to $\her(p'\oplus q)$. Since the selfadjoint $h'=vhv^*$ has the same form   as $h$ (with $p$ replaced by $p'$), we may prove the theorem with  $p'$ in place of $p$. Let us thus assume that $p\leq 1_k$, i.e., $p\in M_k(C(X))$.  Then $\her(p\oplus q)$ is a hereditary subalgebra of $\her(1_k\oplus q)$, and so $h\in \her(1_k\oplus q)$. It will suffice to prove the estimated exponential distance for   $e^{ith}$ in the C*-algebra $\her(1_k\oplus q)$, since an approximation of $e^{ith}$ by products of $m$ commutators in $U_0(\her(p\oplus q))$   yields an approximation for $e^{ith}$ in $U_0(\her(1_k\oplus q))$ within at most the same distance. Note that
\[
	e^{ih}=\begin{pmatrix}
	e^{i\kappa p}+1_k-p& 0\\0 & e^{iw}
	\end{pmatrix}
\]
in $\her(1_k\oplus q)$.

We shall show that for all $0<t_0<k(m+1)\pi$ and $t\geq t_0$ the exponential distance from $e^{ith}$ to $(\{(v,w):v,w\in U_0(\her(1_k\oplus q))\})^{m}$ is $\geq t_0/k$. That the distance from $e^{it}$ to that same set is at least $\min (\frac{t}k, (m+1)\pi)$ for all $t\geq 0$ is easily derived from this.

Fix $0<t_0<k(m+1)\pi$ and $t\geq t_0$. Suppose for the sake of contradiction that the exponential distance from  $e^{ith}$ to the set $(\{(v,w):v,w\in U_0(\her(1_k\oplus q))\})^{m}$ is $<t_0/k$. Then there exist $v_j,w_j\in U_0(\her(1_k\oplus q))$ for $j=1,\ldots,m$ such that the exponential length of 
\[
e^{-ith}(v_1,w_1)\cdots (v_{m},w_{m})
\]
is $<t_0/k$. By  Lemma \ref{ringroselemma} applied with $\theta=\frac{t_0}{k(m+1)}<\pi$, the displayed unitary is expressible as a product  $\prod_{j=1}^{m+1}e^{ih_j} $,  where  $h_j\in \her(1_k\oplus q)_{\sa}$  and 
$\|h_j\|<\frac{t_0}{k(m+1)}$ for all $j$. Omitting $e^{ih_{m+1}}$ we deduce that
\begin{equation}\label{tildeuexps}
\|e^{ith} - (v_1,w_1)\cdots (v_{m},w_{m})\cdot e^{ih_1}\cdots e^{ih_{m}}\|<|e^{i\frac{t_0}{k(m+1)}}-1|.
\end{equation}

Recall that we represent  elements of $\her(1_k\oplus q)$  by a 2x2 matrix
\[
\begin{pmatrix}
a & b\\
c & d
\end{pmatrix},
\]
where $a\in M_k(C(X))$,  $b^t, c\in qM_\infty(C(X))1_k$, and $d\in \her(q)$. Observe that we can regard the first $k$ columns of $b^t$ and $c$ as sections of the vector bundle $\eta_q$ associated to $q$. Now, say,
\[
v_j =\begin{pmatrix}
a_j^{(1)} & b_j^{(1)}\\
c_j^{(1)} & d_j^{(1)}
\end{pmatrix},\quad
w_j=\begin{pmatrix}
a_j^{(2)} & b_j^{(2)}\\
c_j^{(2)} & d_j^{(2)}
\end{pmatrix}\quad
h_j =\begin{pmatrix}
a_j^{(3)} & (b_j^{(3)})^*\\
b_j^{(3)} & c_j^{(3)}
\end{pmatrix}\quad   (j=1,\ldots,m).
\]
Taken all together, the columns of the off-diagonal elements 
\begin{equation}\label{offdiag}
(b_{j}^{(1)})^t,\, c_j^{(1)},\, (b_{j}^{(2)})^t,\, c_j^{(2)}, b_j^{(3)}
\end{equation}
give rise to $5km$ sections of $\eta_q$, equivalently, a single section  of $\eta_q^{\oplus 5km}$.

The set of smooth sections of  $\eta_q^{\oplus 5km}$ that are transversal to the zero section is uniformly dense in the set of all sections, by Thom's transversality theorem (\cite[Theorem I.5]{Thom}). By an approximation, let us perturb the off-diagonal elements \eqref{offdiag} so that the resulting section of $\eta_q^{\oplus 5km}$  is  transversal to the zero section, while \eqref{tildeuexps} is still valid.  Let $Z\subseteq X$ denote the zero set of the off-diagonal elements  \eqref{offdiag}.  The transversality to the zero section implies that $Z$ is an orientable submanifold of $X$ whose homology class $[Z]$ in $H_{*}(X)$ is equal to the Poincare dual of the Euler class of $\eta_q^{\oplus 5km}$ (\cite[Proposition 12.8]{BottTu}). This Euler class is  
$e(q^{\oplus 5km)})=e(q)^{5km}$. 
By assumption, $e(q)^{5km}$ is a nonempty sum of monomials $\alpha_{i_1}\alpha_{i_2}\cdots \alpha_{i_l}$, where $i_j\neq 1$ for all $j$. Using that 
\[
\alpha_{i_1}\alpha_{i_2}\cdots \alpha_{i_l}\cap [X] = [S_{i_1',i_2',\ldots ,i_{n-l}'}],
\]
where the indices on the right side run through the complement of the indices on the left side, we obtain  that $[Z]$ is a sum of homology classes coming from subproducts of spheres of the form
$[S_{1,i_2',\ldots ,i_{n-l}'}]$ (i.e., all the terms include the index $1$). Let us decompose $Z$ into connected components---which are also oriented submanifolds---and pick a connected component $Z'$. Then $[Z']$ is again a sum of homology classes coming from subproducts of spheres, where all the terms include the first sphere, i.e., have the form $[S_{1,i_2,\ldots,i_{l}}]$. 

\emph{Claim}: The projection map $\pi_1\colon X\to S^2$ onto the first sphere maps $Z'\subseteq X$ onto $S^2$. Proof: Suppose it does not. Then $Z'$ is contained in $X'=(S^2\backslash\{x\})\times S_2\times\cdots\times S_2$ for some $x\in S^2$. The inclusion of $X'$ in $X$ induces a map
$H_{*}(X')\to H_{*}(X)$ whose range  is contained in the subgroup
generated by the classes $[S_{j_1,j_2,\cdots ,j_l}]$ with $j_1\neq 1$. This is impossible, since the homology class $[Z']$ belongs to the  range of this map, as the inclusion
of $Z'$ in $X$ factors through the inclusion of $X'$ in $X$.

Having established our claim, let us restrict \eqref{tildeuexps} to the set $Z'$. Since the elements $v_j,w_j,h_j$ are simultaneously of diagonal form on $Z'$, comparing the top left corners we  get that 
\[
\|e^{it\kappa}p+ (1_k-p)- 
(a_1^{(1)},a_1^{(2)})\cdots (a_{m}^{(1)},a_{m}^{(2)})\cdot e^{ia_1^{(3)}}\cdots e^{ia_{m}^{(3)}}|\|<|e^{i\frac{t_0}{k(m+1)}}-1|,
\]
where this relation is now taken in $M_k(C(Z'))$. Here $a_j^{(1)}$ and $a_j^{(2)}$ are unitaries in $M_k(C(Z'))$, and  $a_j^{(3)}$ is selfadjoint such that  $\|a_j^{(3)}\|< \frac{t_0}{k(m+1)}$ for all $j$. This in turn implies that 
\[
e^{it\kappa}p+ (1_k-p)= (a_1^{(1)},a_1^{(2)})\cdots (a_{m}^{(1)},a_{m}^{(2)})\cdot e^{ia_1^{(3)}}\cdots e^{ia_{m}^{(3)}}\cdot e^{ib},
\]
where $b\in M_k(C(Z'))$ is selfadjoint and $\|b\|<\frac{t_0}{k(m+1)}$ (\cite[Proposition 2.4]{Ringrose}). Evaluating at an arbitrary $x\in Z'$ and comparing the determinants of both sides we get
\[
\rank(p)t\kappa(x)-\mathrm{Tr}(b(x))-\sum_{j=1}^{m} \mathrm{Tr}(a_j^{(3)}(x))\in 2\pi \Z
\]
for all $x\in Z'$. Since $Z'$ is connected, 
\[
\rank(p)t\kappa(x)-\mathrm{Tr}(b(x))-\sum_{j=1}^{m} \mathrm{Tr}(a_j^{(3)}(x))
\] 
is constant for $x\in Z'$. This is impossible for $t\geq t_0$. Indeed, on one hand the variation of $\rank(p)t\kappa(x)$ on the set $Z'$ is $2\cdot \rank(p)t\geq 2t_0$, as $\kappa$ maps $Z'$ onto $[-1,1]$. On the other hand, the variation of $\mathrm{Tr}(b(x))+\sum_{j=1}^{m} \mathrm{Tr}(a_j^{(3)}(x))$ on all $X$ is less than 
$2t_0$, as $a_j^{(3)}(x)$ and $b(x)$ are $k\times k$ matrices of norm $<\frac{t_0}{k(m+1)}$. 
\end{proof}

\subsection{Simple inductive limit}
Let us now describe simple inductive limits. We follow the example from \cite[Section 6]{RobertCommutators}, which in turn is based on Villadsen's ``second type'' construction in \cite{villadsen}.

\begin{theorem}\label{mainAH}
There exists a simple unital AH C*-algebra with a unique tracial state and the following property:
For each $m\in \N$  there exists $h\in \overline{[A,A]}\cap A_{\sa}$ of norm $\leq 1$ and such that
\begin{enumerate}[(i)]
\item
$e^{ih}$ is not contained in the norm closure of the set $(\{(v,w):v,w\in U_0(A)\})^{m}$,
\item
the exponential distance from $e^{iCh}$ to the set  $(\{(v,w):v,w\in U_0(A)\})^{m}$ is $\geq (m+1)\pi$
for large enough $C>0$.
\end{enumerate}
\end{theorem}

\begin{proof}
Let $(k_n)_{n=1}^\infty$ be an increasing sequence of natural numbers. We will use this sequence to  construct  the C*-algebra $A$ as an inductive limit. Then,  by letting  $(k_n)_{n=1}^\infty$ grow sufficiently fast, we will show that $A$ has the desired properties.

 For each $n=1,\ldots$, let $X_n=(S^2)^{k_n}$. Let $P_n\in M_{2^n}(C(X_n))$ be the rank one projection $P_n=P^{\otimes k_n}$, 
where $P\in M_2(C(S^2))$ is a  projection  whose first Chern class is a generator  of $H^2(S^2)$. Let us now form the product
$Y_n=\prod_{j=1}^n X_j$. Observe that $Y_n$ is the cartesian product of $(\sum_{j=1}^n k_j)$ 2-dimensional spheres.

Consider the projection $p_n\in M_\infty(C(Y_n))$ defined by
\[
p_n(y)= P_1(x_1)^{\oplus l_1}\oplus P_2(x_2)^{\oplus l_2}\oplus \cdots \oplus P_n(x_n)^{\oplus l_n},
\]
where $y=(x_1,x_2,\dots,x_n)\in Y_n$.  Let us set the numbers $l_n$ recursively  such that $l_1=1$ and $l_{n+1}=\mathrm{rank}(p_{n})$ for $n\geq 1$. (In fact, this yields $l_n=2^{n-1}$ for $n\geq 1$.) Choose for each $n$ a point  $c_n\in Y_n$.

Let $\phi_n\colon \her(p_n)\to \her(p_{n+1})$ be a homomorphism defined as follows:
\[
\phi_n(f)(y,x_{n+1})=
\begin{pmatrix}
f(y) & 0\\
0 & f(c_n)\otimes P_{n+1}(x_{n+1})
\end{pmatrix},
\]
for all $(y,x_{n+1})\in Y_n\times X_{n+1}$. On the right-hand side of this formula we have used the 2x2 matrix notation for
$\her(p_{n+1})=\her(p_n\oplus P_{n+1}^{\oplus l_{n+1}})$. Also, 
$f(c_n)\otimes P_{n+1}$
is regarded  as an element in $\her(P_{n+1}^{\oplus l_{n+1}})$ via the identifications
$\her(P_{n+1}^{\oplus l_{n+1}})\cong M_{l_{n+1}}(\C)\otimes \her(P_{n+1})$
 and $\her(p_n(c_n))\cong M_{l_{n+1}}(\C)$ (recall that $p_{n}$ has rank $l_{n+1}$).

It is known that choosing the points $c_n\in Y_n$ suitably, one can arrange for  the inductive limit C*-algebra $A:=\varinjlim (\her(p_n),\phi_n)$ to be simple and have a unique tracial state (see \cite{villadsen}). The uniqueness of tracial state comes from the fact that if $\sigma$ is a tracial state on $\her(p_{m+n})$ then $\sigma\circ \phi_{m,m+n}$ is an average of $2^n$ traces out of which $2^n-1$  do not depend on $\sigma$ (coming from point evaluations).

Let us now describe how to recursively define the sequence $(k_n)_{n=1}^\infty$ so that the resulting inductive limit C*-algebra has the desired properties.
Define $k_1=1$. Assume that $k_1,\ldots,k_{n}$ have been chosen. Choose $M_n\in \N$ such that 
$[p_n]\leq M_n[1_{Y_n}]$. Now choose   $k_{n+1}\geq 5nM_nl_{n+1}, k_n$, and continue this process ad infinitum.

To prove the theorem, it suffices to find for each $m\in \N$  a selfadjoint element $a_{m}$ in $ \overline{[\her(p_{m+1}),\her(p_{m+1})]}$ of norm $1$ such that 
\begin{enumerate}
\item[(a)]
$e^{ia_{m}}$ is within a distance of at least $\epsilon_m>0$ from any product of $m$ commutators in $U_0(\her(p_{m+1}))$,
\item[(b)]
 $e^{iC_ma_m}$ is within distance $\geq (m+1)\pi$ from any product of $m$ commutators in $U_0(\her(p_{m+1}))$ for a large enough $C_m$.  
\end{enumerate}
 Moreover,  these properties are not destroyed by moving $a_{m}$  along the inductive limit.

Fix $m\in \N$. Recall that $Y_{m}$ is a product of 2-dimensional spheres. Let $\pi_1\colon Y_{m}\to S^2$ be the projection onto the first of the sphere factors in $Y_m$. Let $\kappa\colon Y_{m}\to [-1,1]$ be defined  as in \eqref{kappafunction}. Let $a_m\in \her(p_{m+1})$ be given by
\[
a_m(y,x_{m+1})=\begin{pmatrix}
 \kappa (y)p_{m}(y) & 0\\
0 & -\kappa(y)P_{m+1}^{\oplus l_{m+1}}(x_{m+1})
\end{pmatrix},
\]
for all $(y,x_{m+1})\in Y_{m+1}$,  Since $a_m$ has pointwise zero trace, it is in the kernel of every bounded trace. Hence, $a_m\in \overline{[\her(p_{m+1}),\her(p_{m+1})]}$. Observe also that $\|a_m\|=1$.

Let us verify that the conditions are met to apply Theorem \ref{construction}.
Recall that $M_m\in \N$ is such that
$[p_m]\leq M_m[1_{Y_m}]$. 
Let $\widetilde P_{m+1}\in M_\infty(C(Y_{m+1}))$ be defined as the pull back of $P_{m+1}$ along the projection onto $X_{m+1}$, i.e., 
\begin{equation}\label{Ptilde}
\widetilde{P}_{m+1}(y,x_{m+1})=P_{m+1}(x_{m+1}),\quad \hbox{for $(y,x_{m+1})\in Y_m\times X_{m+1}$.}
\end{equation}
A routine Euler class 
computation   shows that 
\[
e(\widetilde{P}_{m+1})=\sum_{j=1}^{k_{m+1}} \alpha_{j,m+1},
\] 
where we have denoted by $\alpha_{j,m+1}\in H^2(Y_{m+1})$ the  cohomology classes coming from projecting onto the $k_{m+1}$ 2-dimensional spheres in $X_{m+1}=(S^2)^{k_{m+1}}$.  Observe that $e(\widetilde{P}_{m+1})$ belongs to the subring generated by elements in $H^2(Y_{m+1})$ arising from the sphere factors in $X_{m+1}$, and thus omits all the generators coming from the sphere factors in $Y_m$. Moreover, from our choice of the sequence $(k_n)_n$, 
$k_{m+1}\geq 5mM_ml_{m+1}$, and so 
\[
e(\widetilde P_{m+1}^{\oplus l_{m+1}})^{5mM_m}=e(\widetilde{P}_{m+1})^{5mM_ml_{m+1}}\neq 0.
\]
We can therefore apply  Theorem \ref{construction} with $k=M_m$ to conclude that the exponential distance from $e^{ia_m}$ to the set of products of $m$ commutators in  $U_0(\her(p_{m+1}))$ is at least $\epsilon_m=\frac{1}{M_m}$. Moreover, also by   Theorem \ref{construction}, the exponential distance from $e^{iC_ma_m}$, with $C_m=M_m(m+1)\pi$,   to the set of products of $m$ commutators in  $U_0(\her(p_{m+1}))$ is at least $(m+1)\pi$.

Next we show that moving $a_m$ along the inductive limit does not change its  properties. Let $n\in \N$.
Consider  the image of $a_m\in \her(p_{m+1})$ in $\her(p_{m+n})$  by the  connecting homomorphism $\phi_{m+1,m+n}\colon \her(p_{m+1})\to \her(p_{m+n})$. We regard $\her(p_{m+n})$ as $\her(p_{m}\oplus q_{m,n})$, with  
\begin{equation}\label{qmnprojection}
q_{m,n}(x)=P_{m+1}(x_{m+1})^{\oplus l_{m+1}}\oplus \cdots \oplus P_{m+n}(x_{m+n})^{\oplus l_{m+n}},
\end{equation}
for $x=(x_{m+1},\ldots,x_{m+n})\in Y_{n+m}$. Then
\[
\phi_{m+1,m+n}(a_m)(y,x)=
\begin{pmatrix}
\kappa(y)p_{m}(y) & \\
& *
\end{pmatrix},
\]
for $(y,x)\in Y_m\times (X_{m+1}\times\cdots\times X_{m+n})$.
 A routine Euler class computation shows that $e(q_{m,n})$ belongs to a subring of $H^*(Y_{m+n})$ generated by cohomology classes in $H^2(Y_{m+n})$ coming from the  sphere factors in the product $X_{m+1}\times \cdots \times X_{m+n}$. Moreover, from our choice of $(k_n)_n$,
$e(q_{m,n})^{5mM_m}\neq 0$. Hence, applications of Theorem \ref{construction} show that the exponential  distance from
$e^{i\phi_{m+n}(a_m)}$ to the products of $m$ commutators in $U_0(\her(p_{m+n})$ is still $\geq \frac{1}{M_m}$, and  
the distance from $e^{iC_m\phi_{m+n}(a+m)}$ to the same set is $\geq (m+1)\pi$, where $C_m=M_m(m+1)\pi$.
\end{proof}

\begin{proof}[Proof of Theorem \ref{counterexamples}]
Let $A$ be the C*-algebra from Theorem \ref{mainAH}.

(i) The simplicity of $DU_0(A)/Z(DU_0(A))$ has already been proven in Theorem \ref{DUsimple}.
Suppose for the sake of contradiction that $DU_0(A)/Z(DU_0(A))$ has bounded normal generation. 
Since $DU_0(A)/Z(DU_0(A))$ is non-abelian, it follows that there exists $n\in \N$ such that every element of $DU_0(A)/Z(DU_0(A))$ is a product of $n$ commutators. 
This in turn implies that every element of $DU_0(A)$ is a product of $n$ commutators times a scalar unitary.
Since the exponential length of any scalar unitary is at most $\pi$, the exponential distance from any element of $DU_0(A)$ to the set of products of $n$ commutators in $U_0(A)$ is at most $\pi$. Through the density of $DU_0(A)$ in $SU_0(A)$ 
(Theorem \ref{SUprops} (i)), this extends to all elements of $SU_0(A)$. This is in contradiction with Theorem \ref{mainAH} (ii).

(ii) Suppose that $DU_0(A)=SU_0(A)$. Let 
\[
K_m=(\{(u,v):u,v\in U_0(A)\})^m.
\] 
Then $SU_0(A)=\bigcup_{m=1}^\infty \overline{K_m}$, where the closure of $K_m$ is taken in the topology of $SU_0(A)$. Since $SU_0(A)$ is complete, 
$\overline{K_m}$ has non-empty interior for some $m$, by Baire's theorem. It follows that $\overline{K_{2m}}$ contains a neighborhood of the identity of the form $\{e^{ih}:h\in \overline{[A,A]}\cap A_{\sa},\, \|h\|<\epsilon\}$. Hence, for a sufficiently large $n$, $\overline{K_n}$ contains  all $e^{ih}$ with $h\in \overline{[A,A]}\cap A_{\sa}$ and $\|h\|\leq 1$. This contradicts Theorem \ref{mainAH} (i).

(iii) This is clearly a consequence of Theorem \ref{mainAH}.
\end{proof}


\begin{bibdiv}
\begin{biblist}

\bib{ando}{article}{
author={Ando, Hiroshi},
author={Doucha, Michal},
author={Matsuzawa, Yasumichi},
title={Large scale geometry of Banach-Lie groups},
eprint={https://arxiv.org/abs/2011.10376},
date={2020}
}

\bib{APRT}{article}{
author={Antoine, Ramon},
author={Perera, Francesc},
author={Robert, Leonel},
author={Thiel, Hannes},
title={C*-algebras of stable rank one and their Cuntz semigroups},
journal={Duke Math. J. (to appear)}
eprint={https://arxiv.org/abs/1809.03984}
}


\bib{giordano}{article}{
   author={Al-Rawashdeh, Ahmed},
   author={Booth, Andrew},
   author={Giordano, Thierry},
   title={Unitary groups as a complete invariant},
   journal={J. Funct. Anal.},
   volume={262},
   date={2012},
   number={11},
   pages={4711--4730},
}

\bib{ara-mathieu}{book}{
   author={Ara, Pere},
   author={Mathieu, Martin},
   title={Local multipliers of $C^*$-algebras},
   series={Springer Monographs in Mathematics},
   publisher={Springer-Verlag London, Ltd., London},
   date={2003},
   pages={xii+319},
}


\bib{ART}{article}{
   author={Archbold, Robert},
   author={Robert, Leonel},
   author={Tikuisis, Aaron},
   title={The Dixmier property and tracial states for $C^*$-algebras},
   journal={J. Funct. Anal.},
   volume={273},
   date={2017},
   number={8},
   pages={2655--2718},
}



\bib{blackadar}{article}{
   author={Blackadar, Bruce},
   title={Comparison theory for simple $C^*$-algebras},
   conference={
      title={Operator algebras and applications, Vol. 1},
   },
   book={
      series={London Math. Soc. Lecture Note Ser.},
      volume={135},
      publisher={Cambridge Univ. Press, Cambridge},
   },
   date={1988},
   pages={21--54},
}



\bib{BottTu}{article}{
   author={Bott, Raoul},
   author={Tu, Loring W.},
   title={Differential forms in algebraic topology},
   series={Graduate Texts in Mathematics},
   volume={82},
   publisher={Springer-Verlag, New York-Berlin},
   date={1982},
   pages={xiv+331},
}


\bib{bresar}{article}{
   author={Bre\v{s}ar, Matej},
   title={Commuting traces of biadditive mappings, commutativity-preserving
   mappings and Lie mappings},
   journal={Trans. Amer. Math. Soc.},
   volume={335},
   date={1993},
   number={2},
   pages={525--546},
}



\bib{CETWW}{article}{
   author={Castillejos, Jorge},
   author={Evington, Samuel},
   author={Tikuisis, Aaron},
   author={White, Stuart},
   author={Winter, Wilhelm},
   title={Nuclear dimension of simple $\rm C^*$-algebras},
   journal={Invent. Math.},
   volume={224},
   date={2021},
   number={1},
   pages={245--290},
}

\bib{delaHarpe}{article}{
   author={de la Harpe, P.},
   title={Simplicity of the projective unitary groups defined by simple
   factors},
   journal={Comment. Math. Helv.},
   volume={54},
   date={1979},
   number={2},
   pages={334--345},
}


\bib{delaHarpeSkandalis}{article}{
   author={de la Harpe, P.},
   author={Skandalis, G.},
   title={D\'{e}terminant associ\'{e} \`a une trace sur une alg\'{e}bre de Banach},
   language={French, with English summary},
   journal={Ann. Inst. Fourier (Grenoble)},
   volume={34},
   date={1984},
   number={1},
   pages={241--260},
}

\bib{delaHarpeSkandalis2}{article}{
   author={de la Harpe, P.},
   author={Skandalis, G.},
   title={Produits finis de commutateurs dans les $C^\ast$-alg\`ebres},
   language={French, with English summary},
   journal={Ann. Inst. Fourier (Grenoble)},
   volume={34},
   date={1984},
   number={4},
   pages={169--202},
}


\bib{delaHarpe-Skandalis3}{article}{
   author={de la Harpe, P.},
   author={Skandalis, G.},
   title={Sur la simplicit\'e essentielle du groupe des inversibles et du
   groupe unitaire dans une $C\sp \ast$-alg\`ebre simple},
   language={French, with English summary},
   journal={J. Funct. Anal.},
   volume={62},
   date={1985},
   number={3},
   pages={354--378},
}


\bib{dowerkthesis}{thesis}{
author= {Dowerk, Philip A.},
title = {Algebraic and topological properties of unitary groups of II$_1$ factors},
type={Ph.D. Thesis},
address={University of Leipzig},
date={2015},
}


\bib{DowerkLeMaitre}{article}{
	author={Dowerk, Philip A.},
	author={Le Ma\^{\i}tre, Fran\c{c}ois},
	title={Bounded normal generation is not equivalent to topological bounded
		normal generation},
	journal={Extracta Math.},
	volume={34},
	date={2019},
	number={1},
	pages={85--97},
}

\bib{DowerkThom}{article}{
author={Dowerk, Philip A.},
author={Thom, Andreas},
title={Bounded normal generation and invariant automatic continuity},
journal={Adv. Math.},
volume={346},
date={2019},
pages={124--169},
}

\bib{DowerkThom2}{article}{ 
   author={Dowerk, Philip A.},
   author={Thom, Andreas},
   title={Bounded normal generation for projective unitary groups of certain
   infinite operator algebras},
   journal={Int. Math. Res. Not. IMRN},
   date={2018},
   number={24},
   pages={7642--7654},
}




\bib{elliott-rordam0}{article}{
   author={Elliott, George A.},
   author={R\o rdam, Mikael},
   title={The automorphism group of the irrational rotation $C^*$-algebra},
   journal={Comm. Math. Phys.},
   volume={155},
   date={1993},
   number={1},
   pages={3--26},
}


\bib{elliott-rordam}{article}{
   author={Elliott, George A.},
   author={R\o rdam, Mikael},
   title={Perturbation of Hausdorff moment sequences, and an application to
   the theory of $C^*$-algebras of real rank zero},
   conference={
      title={Operator Algebras: The Abel Symposium 2004},
   },
   book={
      series={Abel Symp.},
      volume={1},
      publisher={Springer, Berlin},
   },
   date={2006},
   pages={97--115},
}

\bib{EGLN2}{article}{
author={Elliott, George A.},
	author={Gong, Guihua},
	author={Lin, Huaxin},
	author={Niu, Zhuang},
	title={On the classification of simple unital C*-algebras with finite decomposition rank. II},
	eprint={https://arxiv.org/abs/1507.03437},
	date={2016}
}

\bib{DUR}{article}{
	author={Gong, Guihua},
	author={Lin, Huaxin},
	author={Xue, Yifeng},
	title={Determinant rank of $C^*$-algebras},
	journal={Pacific J. Math.},
	volume={274},
	date={2015},
	number={2},
	pages={405--436},
}	

\bib{GLN1}{article}{
	author={Gong, Guihua},
	author={Lin, Huaxin},
	author={Niu, Zhuang},
	title={Classification of finite simple amenable $\mathcal Z$-stable C*-algebras, I. C*-algebras with generalized tracial rank one},
	journal={C. R. Math. Acad. Sci. Soc. R. Can.},
	volume={42},
	number={3},
	date={2020},
	pages={63--450},
}






\bib{haagerup}{article}{
   author={Haagerup, Uffe},
   title={Quasitraces on exact $C^*$-algebras are traces},
   language={English, with English and French summaries},
   journal={C. R. Math. Acad. Sci. Soc. R. Can.},
   volume={36},
   date={2014},
   number={2-3},
   pages={67--92},
}


\bib{herstein}{article}{
   author={Herstein, I. N.},
   title={On the Lie structure of an associative ring},
   journal={J. Algebra},
   volume={14},
   date={1970},
   pages={561--571},
}

\bib{hofmann-morris}{book}{
   author={Hofmann, Karl H.},
   author={Morris, Sidney A.},
   title={The structure of compact groups},
   series={De Gruyter Studies in Mathematics},
   volume={25},
   edition={4},
   note={A primer for the student---a handbook for the expert},
   publisher={De Gruyter, Berlin},
   date={[2020] \copyright 2020},
   pages={1006},
   isbn={978-3-11-069599-1},
   isbn={978-3-11-069595-3},
   isbn={978-3-11-069601-1},
}


\bib{jacelon}{article}{
   author={Jacelon, Bhishan},
   author={Strung, Karen R.},
   author={Toms, Andrew S.},
   title={Unitary orbits of self-adjoint operators in simple
   $Z$-stable $\rm C^*$-algebras},
   journal={J. Funct. Anal.},
   volume={269},
   date={2015},
   number={10},
   pages={3304--3315},
}



\bib{kad1}{article}{
				author={Kadison, R. V.},
				title={Infinite unitary groups},
				journal={Trans. Amer. Math. Soc.},
				volume={72},
				date={1952},
				pages={386--399},
}





\bib{kallman}{article}{
   author={Atim, Alexandru G.},
   author={Kallman, Robert R.},
   title={The infinite unitary and related groups are algebraically
   determined Polish groups},
   journal={Topology Appl.},
   volume={159},
   date={2012},
   number={12},
   pages={2831--2840},
}


\bib{KirchbergRordam}{article}{
author={Kirchberg, Eberhard},
author={R\o rdam, Mikael},
title={Central sequence $C^*$-algebras and tensorial absorption of the
Jiang-Su algebra},
journal={J. Reine Angew. Math.},
volume={695},
date={2014},
pages={175--214},
}

\bib{KirchbergRordam2}{article}{
author={Kirchberg, Eberhard},
author={R\o rdam, Mikael},
title={Infinite non-simple $C^*$-algebras: absorbing the Cuntz algebras
	$\scr O_\infty$},
journal={Adv. Math.},
volume={167},
date={2002},
number={2},
pages={195--264},
}


\bib{kleinecke}{article}{
   author={Kleinecke, David C.},
   title={On operator commutators},
   journal={Proc. Amer. Math. Soc.},
   volume={8},
   date={1957},
   pages={535--536},
}

\bib{XinLi}{article}{
   author={Li, Xin},
   title={Every classifiable simple $\rm C^*$-algebra has a Cartan
   subalgebra},
   journal={Invent. Math.},
   volume={219},
   date={2020},
   number={2},
   pages={653--699},
}



\bib{LinRR0}{article}{
author={Lin, Hua Xin},
title={Exponential rank of $C^*$-algebras with real rank zero and the
Brown-Pedersen conjectures},
journal={J. Funct. Anal.},
volume={114},
date={1993},
number={1},
pages={1--11},
}


\bib{LinTR1}{article}{
author={Lin, Huaxin},
title={Exponentials in simple $Z$-stable $C^*$-algebras},
journal={J. Funct. Anal.},
volume={266},
date={2014},
number={2},
pages={754--791},
}



\bib{ker-det}{article}{
author={Ng, Ping Wong},
author={Robert, Leonel},
title={The kernel of the determinant map on pure C*-algebras},
journal={Houston J. Math.},
volume={43},
date={2017},
number={1},
pages={139--168},
}


\bib{marcoux06}{article}{
   author={Marcoux, L. W.},
   title={Sums of small number of commutators},
   journal={J. Operator Theory},
   volume={56},
   date={2006},
   number={1},
   pages={111--142},
}


\bib{marcoux09}{article}{
   author={Marcoux, L. W.},
   title={Projections, commutators and Lie ideals in $C^*$-algebras},
   journal={Math. Proc. R. Ir. Acad.},
   volume={110A},
   date={2010},
   number={1},
   pages={31--55},
}


\bib{NgRobertpure}{article}{
author={Ng, Ping Wong},
author={Robert, Leonel},
title={Sums of commutators in pure $\rm C^*$-algebras},
journal={M\"{u}nster J. Math.},
volume={9},
date={2016},
number={1},
pages={121--154},
}


\bib{ng-ruiz}{article}{
   author={Ng, P. W.},
   author={Ruiz, E.},
   title={The automorphism group of a simple $\mathcal Z$-stable $\rm C^*$-algebra},
   journal={Trans. Amer. Math. Soc.},
   volume={365},
   date={2013},
   number={8},
   pages={4081--4120},
}


\bib{phillips}{article}{
   author={Phillips, N. Christopher},
   title={Approximation by unitaries with finite spectrum in purely infinite
   $C^*$-algebras},
   journal={J. Funct. Anal.},
   volume={120},
   date={1994},
   number={1},
   pages={98--106},
}

\bib{phillipsHowMany}{article}{
   author={Phillips, N. Christopher},
   title={How many exponentials?},
   journal={Amer. J. Math.},
   volume={116},
   date={1994},
   number={6},
   pages={1513--1543},
}

\bib{phillipsFU}{article}{
   author={Phillips, N. Christopher},
   title={Simple $C^*$-algebras with the property weak (FU)},
   journal={Math. Scand.},
   volume={69},
   date={1991},
   number={1},
   pages={127--151},
}


\bib{Pop}{article}{
author={Pop, Ciprian},
title={Finite sums of commutators},
journal={Proc. Amer. Math. Soc.},
volume={130},
date={2002},
number={10},
pages={3039--3041},
}

\bib{rieffel1}{article}{
   author={Rieffel, Marc A.},
   title={Dimension and stable rank in the $K$-theory of
   $C\sp{\ast}$-algebras},
   journal={Proc. London Math. Soc. (3)},
   volume={46},
   date={1983},
   number={2},
   pages={301--333},
}


\bib{rieffel2}{article}{
	author={Rieffel, Marc A.},
	title={The homotopy groups of the unitary groups of noncommutative tori},
	journal={J. Operator Theory},
	volume={17},
	date={1987},
	number={2},
	pages={237--254},
}


\bib{Ringrose}{article}{
   author={Ringrose, J. R.},
   title={Exponential length and exponential rank in $C^*$-algebras},
   journal={Proc. Roy. Soc. Edinburgh Sect. A},
   volume={121},
   date={1992},
   number={1-2},
   pages={55--71},
}



\bib{RobertCommutators}{article}{
   author={Robert, Leonel},
   title={Nuclear dimension and sums of commutators},
   journal={Indiana Univ. Math. J.},
   volume={64},
   date={2015},
   number={2},
   pages={559--576},
}


\bib{RobertLie}{article}{
author={Robert, Leonel},
title={On the Lie ideals of $C^*$-algebras},
journal={J. Operator Theory},
volume={75},
date={2016},
number={2},
pages={387--408},
}		


\bib{RobertNormal}{article}{
author={Robert, Leonel},
title={Normal subgroups of invertibles and of unitaries in a ${\rm
	C}^*$-algebra},
journal={J. Reine Angew. Math.},
volume={756},
date={2019},
pages={285--319},
}


\bib{RordamZ}{article}{
author={R\o rdam, Mikael},
title={The stable and the real rank of $\scr Z$-absorbing $C^*$-algebras},
journal={Internat. J. Math.},
volume={15},
date={2004},
number={10},
pages={1065--1084},
}


\bib{rosendalOB}{article}{
   author={Rosendal, Christian},
   title={A topological version of the Bergman property},
   journal={Forum Math.},
   volume={21},
   date={2009},
   number={2},
   pages={299--332},
}

\bib{Thiel}{article}{
author={Thiel, Hannes},
title={Ranks of operators in simple $C^*$-algebras with stable rank one},
journal={Comm. Math. Phys.},
volume={377},
date={2020},
number={1},
pages={37--76},
}



\bib{Thom}{article}{
   author={Thom, Ren\'{e}},
   title={Quelques propri\'{e}t\'{e}s globales des vari\'{e}t\'{e}s diff\'{e}rentiables},
   language={French},
   journal={Comment. Math. Helv.},
   volume={28},
   date={1954},
   pages={17--86},
}


\bib{thomsen}{article}{
   author={Thomsen, Klaus},
   title={Finite sums and products of commutators in inductive limit
   $C^\ast$-algebras},
   language={English, with English and French summaries},
   journal={Ann. Inst. Fourier (Grenoble)},
   volume={43},
   date={1993},
   number={1},
   pages={225--249},
}


\bib{TWW}{article}{
   author={Tikuisis, Aaron},
   author={White, Stuart},
   author={Winter, Wilhelm},
   title={Quasidiagonality of nuclear $C^\ast$-algebras},
   journal={Ann. of Math. (2)},
   volume={185},
   date={2017},
   number={1},
   pages={229--284},
}




\bib{villadsen}{article}{
   author={Villadsen, Jesper},
   title={On the stable rank of simple $C^\ast$-algebras},
   journal={J. Amer. Math. Soc.},
   volume={12},
   date={1999},
   number={4},
   pages={1091--1102},
}

\bib{zhang}{article}{
   author={Zhang, Shuang},
   title={A Riesz decomposition property and ideal structure of multiplier
   algebras},
   journal={J. Operator Theory},
   volume={24},
   date={1990},
   number={2},
   pages={209--225},
}

\end{biblist}
\end{bibdiv}
\end{document}